\documentclass[a4paper,11pt]{amsart}
\usepackage{graphicx} 
\usepackage{amsfonts, amssymb, amsmath}
\usepackage{amsthm}
\usepackage{mathtools}
\usepackage{mathrsfs}
\usepackage[margin=1in]{geometry}
\usepackage{enumerate}
\usepackage[all]{xy}
\usepackage{color, xcolor}

\usepackage{tikz, tikz-cd}
\usetikzlibrary{3d}
\usepackage{caption}
\usepackage{subcaption}
\usepackage[all]{xy}

\newcommand{\PP}{\mathbb{P}}

\newcommand{\QQ}{\mathbb{Q}}

\newcommand{\cX}{\mathcal{X}}

\newcommand{\n}{\noindent}

\theoremstyle{definition}
\newtheorem*{Def}{Definition}
\newtheorem{Thm}{Theorem}[section]
\newtheorem{Prop}[Thm]{Proposition}
\newtheorem{Lem}[Thm]{Lemma}
\newtheorem{Rmk}[Thm]{Remark}
\newtheorem{Cor}[Thm]{Corollary}
\newtheorem{Ex}[Thm]{Example}

\usepackage{epic}

\title[Weighted Projective Smoothings to Fano Threefolds]
{Weighted Projective Spaces Admitting \(\QQ\)-Gorenstein \\Smoothings
to Fano Threefolds with Picard Number One}

\author{Jungkai Alfred Chen and Yongnam Lee}
\date{September 2026}

\address{Department of Mathematics, National Taiwan University, No. 1, Sec. 4, Roosevelt Rd., Taipei 10617, Taiwan}
\email{jkchen@ntu.edu.tw}

\address{Center for Complex Geometry, Institute for Basic Science (IBS), 55 Expo-ro, Yuseong-gu, Daejeon 34126, Korea}
\email{ynlee@ibs.re.kr}


\begin{document}

\begin{abstract}
We study well-formed weighted projective threefolds
$X=\PP(w_0,w_1,w_2,w_3)$ that admit a
$\QQ$-Gorenstein smoothing to a smooth Fano threefold of Picard number one.  For del Pezzo threefolds $V_d$ and prime Fano threefolds $Y_g$, we derive three necessary numerical and local conditions: the anticanonical volume equation, an identity obtained from the linear term of the anticanonical Hilbert polynomial, and a global section condition for
smoothing the transversal $A$-singularities along coordinate curves.  A computer search using these conditions determines all numerical candidates, apart from the known infinite family for $V_5$, with $w_0+w_1+w_2+w_3\leq 5000$.  We construct $\QQ$-Gorenstein smoothings to $V_1, V_2, Y_6, Y_{10},$ and $Y_{12}$ from numerical candidates of weighted projective threefolds, and prove that $\PP(2,5,8,25)$, although a numerical candidate for $V_4$, is not $\QQ$-Gorenstein smoothable.  To identify the smooth fibres, we also give a vanishing-cycle criterion ensuring that Picard number
one is preserved under the smoothing.
\end{abstract}

\maketitle

\section{Introduction}

Throughout the paper we work over the complex numbers.  Weighted projective spaces occur naturally as toric degenerations and provide an explicit setting in which to study $\QQ$-Gorenstein deformation theory.  Our starting
point is the following classification problem:
Which well-formed weighted projective threefolds
$X=\PP(w_0,w_1,w_2,w_3)$ admit a \(\QQ\)-Gorenstein smoothing whose general fibre is a smooth
Fano threefold $Y$ with $\rho(Y)=1$?

This paper is a sequel to \cite{CL26}, where we considered the case $Y=\PP^3$.  In that case two infinite families are known.  The first
consists of the weighted projective spaces
\[\PP(a^2,b^2,c^2,abc),\qquad a^2+b^2+c^2=3abc,\]
which we call weighted projective spaces of \(\PP^2\)-type.  The second consists of
\[\PP(2a^2,b^2,c^2,4abc),\qquad 2a^2+b^2+c^2=4abc,\]
which are of $Q$-type.  DeVleming conjectured that these are the only weighted projective threefolds admitting a $\QQ$-Gorenstein smoothing to $\PP^3$ (cf. \cite{DeV22}); our previous paper provides evidence for this conjecture.

There are analogous infinite families for the smooth quadric $Q^3$ and the degree-five del Pezzo threefold $V_5$.  If
$a^2+b^2+2c^2=4abc$, then $\PP(a^2,b^2,2c^2)$ smooths to a quadric surface \cite{HP10}, and the relative-cone construction for the half-anticanonical polarization \cite[Section~3]{DLT} gives a smoothing
\[\PP(a^2,b^2,2c^2,2abc)\rightsquigarrow Q^3.\] 

Similarly, if $a^2+b^2+5c^2=5abc$,
then $\PP(a^2,b^2,5c^2)$ smooths to a toric del Pezzo surface of degree five \cite{HP10}, and the same cone construction \cite[Section~3]{DLT} gives
\[\PP(a^2,b^2,5c^2,5abc)\rightsquigarrow V_5.\]

Outside these cone constructions, relatively few weighted projective
threefolds are known to smooth to Fano threefolds of Picard number one.  This
leads to two natural questions.

\begin{enumerate}
\item[\textup{(Q1)}] Does every deformation family of smooth Fano
threefolds with Picard number one contain a member admitting a
$\QQ$-Gorenstein degeneration to a weighted projective threefold?

\item[\textup{(Q2)}] Apart from $\PP^3$, $Q^3$, and $V_5$, are there only finitely many weighted projective threefolds that admit such smoothings within each deformation family?
\end{enumerate}

Our main purpose is to develop effective numerical and local criteria for
these questions and to construct new examples relevant to \textup{(Q1)}.
Let
\[X=\PP(w_0,w_1,w_2,w_3),\qquad
S=\sum_{i=0}^3w_i,\qquad P=\prod_{i=0}^3w_i.\]
If $X$ admits a $\QQ$-Gorenstein smoothing to a smooth Fano
threefold $Y$, deformation invariance of the anticanonical volume gives
\[S^3=(-K_Y)^3P.                      \tag{V}\]
We call this the {\it volume equation}. Consequently,
\[S^3= \left\{\begin{array}{ll}
8dP, &\text{if }Y=V_d; \\
(2g-2)P, &\text{if }Y=Y_g. \end{array} \right.\]
Here $-K_{V_d}=2H$, $H^3=d$, whereas $-K_{Y_g}=H$ and
$H^3=2g-2$.

The volume equation is necessary but far from sufficient.  Put
$g_{ij}=\gcd(w_i,w_j)$, and let $w_k,w_\ell$ denote the complementary weights.  Assuming that every codimension-two singularity is of $A$-type, comparison of the linear coefficients of the anticanonical Hilbert polynomials gives the identities
\[\sum_{i<j}(w_i-w_j)^2
 +3\left(\frac{8}{d}-1\right)S^2
=8\sum_{i<j}(g_{ij}^2-1)w_kw_\ell                 \tag{$\mathrm{HA}_{V_d}$}\]
for $V_d$, and
\[\sum_{i<j}(w_i-w_j)^2
 +\left(\frac{96}{g-1}-3\right)S^2
=8\sum_{i<j}(g_{ij}^2-1)w_kw_\ell                 \tag{$\mathrm{HA}_{Y_g}$}\]
for $Y_g$.  We refer to these equations collectively as the
\emph{Hilbert identity}.

There is a further local-to-global restriction.  Suppose that
\[C_{ij}=\{x_k=x_\ell=0\}\]
has generic transversal singularity of type \(A_{g_{ij}-1}\).  We show that smoothing this curve requires $w_k+w_\ell=g_{ij}\delta$ for some positive
integer $\delta$, together with
\[g_{ij}\delta
 \in\left\langle\frac{w_i}{g_{ij}},\frac{w_j}{g_{ij}}\right\rangle_{\mathbb Z_{\geq0}}
\quad\text{or}\quad
(g_{ij}-1)\delta
 \in\left\langle\frac{w_i}{g_{ij}},\frac{w_j}{g_{ij}}\right\rangle_{\mathbb Z_{\geq0}}.
\tag{S}\]
This is a global section condition for the first-order smoothing parameter along $C_{ij}$, which we call the {\it smoothing section condition}.  It imposes further conditions other than  the volume equation and the Hilbert identity.

We call $X$ a \emph{numerical candidate} for $V_d$, respectively
$Y_g$, if it satisfies the appropriate volume equation (V) and Hilbert identity (HA), as well as the smoothing-section condition \textup{(S)}.  A computer search yields all numerical candidates with $S\leq5000$, with $V_5$ omitted because of the infinite family described above.  The result
is shown in Table~\ref{tab:candidates}.

\begin{table}[htbp]
\centering
\begin{tabular}{c|l||c|l}
Family & Numerical candidates & Family & Numerical candidates \\
\hline
\(V_1\) & \(\PP(5,12,18,25)\)
  & \(Y_2\) & none \\
\(V_2\) & \(\PP(3,4,8,9)\)
  & \(Y_3\) & none \\
 & \(\PP(3,14,18,49)\)
  & \(Y_4\) & none \\
\(V_3\) & \(\PP(2,3,3,4)\)
  & \(Y_5\) & \(\PP(5,12,18,25)\) \\
 & \(\PP(2,12,21,49)\)
  & \(Y_6\) & \(\PP(5,6,9,10)\) \\
\(V_4\) & \(\PP(2,5,8,25)\)
  & & \(\PP(6,45,170,289)\) \\
 & & \(Y_7\) & \(\PP(3,12,20,25)\) \\
 & & \(Y_8\) & none \\
 & & \(Y_9\) & \(\PP(3,4,8,9)\) \\
 & & & \(\PP(3,14,18,49)\) \\
 & & \(Y_{10}\) & \(\PP(2,18,25,45)\) \\
 & & \(Y_{12}\) & \(\PP(2,9,22,33)\) \\
 & & & \(\PP(8,25,55,352)\)
\end{tabular}
\caption{Numerical candidates with \(S\leq5000\), excluding \(V_5\).}
\label{tab:candidates}
\end{table}

DeVleming--Li--Torres \cite{DLT} proved that
$\PP(2,3,3,4)$ and $\PP(2,12,21,49)$ admit
$\QQ$-Gorenstein smoothings to the cubic threefold $V_3$.  For several of the remaining candidates we construct explicit or local-to-global smoothings.  Our main theorem is the following.

\begin{Thm}\label{thm:main}
The following \(\QQ\)-Gorenstein smoothings exist:
\begin{enumerate}
\item $\PP(5,12,18,25)\rightsquigarrow V_1$;
\item $\PP(3,4,8,9)\rightsquigarrow V_2$ and
      $\PP(3,14,18,49)\rightsquigarrow V_2$;
\item $\PP(5,6,9,10)\rightsquigarrow Y_6$;
\item $\PP(2,18,25,45)\rightsquigarrow Y_{10}$;
\item $\PP(2,9,22,33)\rightsquigarrow Y_{12}$.
\end{enumerate}
\end{Thm}

The Picard number of a smooth fibre is not determined by the weighted projective special fibre alone: a single weighted projective threefold can have different smoothing components whose general fibres have different Picard numbers.  We therefore use nearby and vanishing cycles to give a
criterion for $\rho(Y)=1$.  Roughly, if the degree-two vanishing-cycle group vanishes, then specialization induces a surjection
\[H^2(X,\QQ)\longrightarrow H^2(Y,\QQ).\]
Since a weighted projective threefold has one-dimensional rational second cohomology, this forces $\rho(Y)=1$.  In the examples of
Theorem~\ref{thm:main}, the required vanishing is established by studying the monodromy of the $A$-type vanishing-cycle local systems and their behaviour at the endpoints of the singular curves.

We also obtain two complementary obstruction results.  First,
$\PP(2,5,8,25)$, although a numerical candidate for $V_4$, admits no $\QQ$-Gorenstein smoothing; the obstruction comes from the incompatibility of the curve smoothing with an endpoint singularity.  Second, for a natural partial smoothing associated with the $Y_7$-candidate $\PP(3,12,20,25)$, we compute a six-dimensional group $H^1(T^1_{\mathrm{QG}})$.  The latter computation exhibits a globalization difficulty, although it does not by itself prove that the original weighted projective space is non-smoothable.

Our problem is more restrictive than the existence of an arbitrary toric degeneration.  Ilten--Lewis--Przyjalkowski proved that every smooth Fano threefold of Picard rank one admits a flat degeneration to a possibly singular toric Fano threefold \cite[Theorem~3.1]{ILP}.  Their special fibres,
however, need not be weighted projective spaces, and the total spaces of the degenerations need not be $\QQ$-Gorenstein.  Galkin's classification of small toric degenerations gives a further indication of how restrictive
our problem is \cite{Gal}.  Thus neither result directly supplies the weighted-projective $\QQ$-Gorenstein degenerations considered here.

The paper is organized as follows.  In Section~2 we derive the Hilbert identity and the smoothing-section condition, study compatibility at the endpoints of singular curves, and prove the non-smoothability of $\PP(2,5,8,25)$.  Section~3 develops the nearby- and vanishing-cycle
criterion used to control the Picard number of a smooth fibre.  In Section~4
we construct the smoothings in Theorem~\ref{thm:main}, identify their general
fibres, and discuss the remaining obstruction for
$\PP(3,12,20,25)$.

\subsection*{Acknowledgements}
J. Chen is partially supported by National Center for Theoretical Sciences and National Science and Technology Council (115-2811-M-002-126) of Taiwan. Y. Lee is supported by the Institute for Basic Science
(IBS-R032-D1). Y. Lee would like to thank the Department of Mathematics and NCTS in National Taiwan University for the hospitality during his visit in 2026. The authors would like to thank Kristin DeVleming for some useful discussion and comments.

\section{Hilbert  identity and smoothing section condition}

Let $X=\PP(w_0,w_1,w_2,w_3)$ be a well-formed weighted projective space. Since it is well-formed, we have $\operatorname{gcd}(w_i, w_j, w_k)=1$ for distinct $\{i,j,k\} \subset \{0,1,2,3\}$. Details on weighted projective spaces can be found in \cite{BR86, Dol82, IF00}.
In this paper, we will always assume well-formed. We assume that $X$ is a $\QQ$-Gorenstein degeneration of a Fano threefold,  that is, there exists a flat projective  morphism $\pi: \cX\to\Delta$ over a DVR with the central fiber $X$ such that $\pi^{-1}(t)=X_t$ is a smooth Fano threefold for $t\ne 0$ and $K_{\cX/\Delta}$ is $\QQ$-Cartier. 

Let $S=w_0+w_1+w_2+w_3,\ P=w_0w_1w_2w_3.$
Then $(-K_X)^3=\frac{S^3}{P}.$ If $X$ admits a $\QQ$-Gorenstein smoothing to a smooth Fano threefold $Y$, constancy of anticanonical volume gives
$S^3=(-K_Y)^3P.$ For a smooth Fano threefold $Y$ with $\rho(Y)=1$, write $-K_Y=rH,$ where $H$ generates $\operatorname{Pic}(Y)$ and $r$ is the Fano index. Then
$(-K_Y)^3=r^3H^3,$ so the general volume equation is $S^3=r^3H^3\,P$.

{\it Index 4:} The only possibility is $Y=\PP^3$ with $(-K_Y)^3=64$. Thus $S^3=64P.$

{\it Index 3:} The only possibility is a smooth quadric $Q^3\subset\PP^4$ with $-K_{Q^3}=3H,\ H^3=2,$ so $S^3=54P$.

{\it Index 2:} These are the del Pezzo threefolds $V_d$, where
\[-K_{V_d}=2H,\qquad H^3=d,\qquad 1\le d\le5.\]
Therefore $(-K_{V_d})^3=8d$ and $S^3=8dP$ with $ 1 \le d \le 5$.

{\it Index 1:} These are prime Fano threefolds of genus
$g\in\{2,3,4,5,6,7,8,9,10,12\},$ where
\[-K_Y=H,\qquad H^3=2g-2.\]
Thus $S^3=(2g-2)P$.

Let $S=\sum_{i=0}^3w_i$ and $L=\operatorname{lcm}(w_0,w_1,w_2,w_3).$ Let
\[r=\frac{\operatorname{lcm}(w_0,w_1,w_2,w_3)}
        {\gcd(\operatorname{lcm}(w_0,w_1,w_2,w_3),S)}.\]
Then $r$ is the Cartier index of $K_X$, and $rK_X$ is Cartier exactly when
\[\mathcal O_X(-rS)\ \text{is Cartier}
\quad \text{if and only if} \quad
L\mid rS.\]

Let $H_X(m)=\chi\!\left(X,\mathcal O_X(-mrK_X)\right)$ be the Hilbert polynomial associated with the Cartier line bundle $\mathcal O_X(-rK_X)$.
Then $H_X(m)$ satisfies the following Hilbert coefficient identity. This constancy of the coefficients of $H_X(m)$ coming from the anticanonical $\QQ$-line bundle gives the fundamental equation.

\[[t^{mrS}]\frac{1}{(1-t^{w_0})(1-t^{w_1})(1-t^{w_2})(1-t^{w_3})}
=h^0\!\left(Y,\mathcal O_Y(-mrK_Y)\right),\]
when \(\PP(w_0,w_1,w_2,w_3)\) admits a \(\QQ\)-Gorenstein smoothing to a Fano 3-fold \(Y\).
Here the notation $[t^N]H(t)$ means “take the coefficient of $t^N$ in the power-series expansion of $H(t)$.”

\begin{Rmk}
The cubic and quadratic coefficients of the anticanonical Hilbert polynomial are equivalent.
For a threefold, Riemann–Roch begins
\[\chi(X,\mathcal O_X(mD))=
\frac{D^3}{6}m^3
-\frac{K_XD^2}{4}m^2+\cdots.\]
Taking $D=-rK_X$ makes both coefficients multiples of $(-K_X)^3$:
\[D^3=r^3(-K_X)^3,\qquad
-K_XD^2=r^2(-K_X)^3.\]
Therefore both coefficients contain the same numerical invariant, the anticanonical volume.
\end{Rmk}

For every singular edge $(w_i,w_j)$, put $g_{ij}=\gcd(w_i,w_j)$,with complementary weights $(w_k,w_\ell)$. The next lemma shows when $X=\mathbb{P}(w_0, w_1, w_2, w_3)$ admits a smoothing to a Fano threefold $Y$.
It is proved for $Y=\PP^3$ in \cite[Lemma 2.2]{CL26}. Proof is basically the same for any Fano threefold $Y$.

\begin{Lem}\label{smoothing} Suppose $X$ admits a $\QQ$-Gorenstein smoothing to a smooth threefold.
\begin{enumerate}[(i)]
\item If $w_i >1$, then $g_{ij} >1$ for some $j\ne i$.
\item If $g_{ij}=d >1$ and $d$ is square free then $X$ has only $A_{d-1}$ singularities along the curve $C_{ij}=\{x_k=x_\ell=0\}\subset X$.
\item Suppose the generic transverse singularity is $A_{g_{ij}-1}$ along the curve $C_{ij}=\{x_k=x_\ell=0\}\subset X$. Then 
$w_k+w_\ell=g_{ij}\delta$ for some positive integer $\delta$, and 
\[g_{ij}\delta\in\langle \frac{w_i}{g_{ij}},\ \frac{w_j}{g_{ij}}\rangle_{\mathbb Z_{\ge0}}
\quad\text{or}\quad
(g_{ij}-1)\delta\in\langle \frac{w_i}{g_{ij}}, \ \frac{w_j}{g_{ij}}\rangle_{\mathbb Z_{\ge0}}.\]

In particular, if $(g_{ij}=2$, then $w_k+w_\ell=2\delta\in\langle \frac{w_i}{g},\ \frac{w_j}{g_{ij}}\rangle_{\mathbb Z_{\ge0}}$.
\end{enumerate}
\end{Lem}

\begin{Def}
We say that every codimension-two singularity is $A$-type in $X=\PP(w_0,w_1,w_2,w_3)$ if it satisfies the following:
Whenever 
$g_{ij} > 1,$ $X$ has $A_{g_{ij}-1}$-singularities along the corresponding coordinate curve $C_{ij}=\{x_k=x_\ell=0\}\subset X, \ \{k,\ell\}=\{0,1,2,3\}\setminus\{i,j\}$.
\end{Def}

This is equivalent to $w_k^{-1}w_{\ell} \equiv-1\pmod {g_{ij}}.$ Thus every codimension-two singularity is $A$-type means:
$g_{ij}\mid w_k+w_\ell
\quad\text{for every pair }i,j\text{ with }g_{ij}>1$.

\begin{Lem}\label{linear}
Suppose that $X=\PP(w_0,w_1,w_2,w_3)$ admits a $\QQ$-Gorenstein smoothing to a smooth threefold and every codimension-two singularity in $X$ is $A$-type. Then constancy of the anticanonical Hilbert polynomial gives
\[\sum_{i<j}(w_i-w_j)^2+3\left(\frac8d-1\right)S^2
=8\sum_{i<j}(g_{ij}^2-1)w_kw_\ell.
\tag{$\mathrm{HA}_{V_d}$}\]
for del Pezzo threefold $V_d$.

 \[\sum_{i<j}(w_i-w_j)^2+
\left(\frac{96}{g-1}-3\right)S^2
=8\sum_{i<j}(g_{ij}^2-1)w_kw_\ell.
\tag{$\mathrm{HA}_{Y_g}$}\]
for prime Fano threefold with genus $g$.
\end{Lem}

\begin{proof}
The proof follows the same idea as in \cite[Lemma 2.6]{CL26}. Let $V_d$ be a smooth del Pezzo threefold of degree $d$, so
\[-K_{V_d}=2H,\qquad H^3=d,\qquad (-K_{V_d})^3=8d,\]
where $1\le d\le5$.

Riemann–Roch formula on $V_d$ gives
\[\chi(V_d,-mK_{V_d})=\frac{4d}{3}m^3+2dm^2+\frac{2(d+3)}3m+1, \tag{1}\]
using $(-K_{V_d})^3=8d$ and $(-K_{V_d})c_2(V_d)=24$.

For the weighted projective space, the contribution from the pole $t=1$ of $\frac1{\prod_i(1-t^{w_i})}$
has linear coefficient
$\frac{S\left(S^2+\sum_{i<j}w_iw_j\right)}{12P}$. A singular coordinate curve corresponding to $g_{ij}>1$ contributes
$\frac{S}{12P}(g_{ij}^2-1)w_kw_\ell$ to the linear coefficient. The $A$-type assumption is used through
\[w_k+w_\ell\equiv0\pmod{g_{ij}},\]
which gives
\[\sum_{q=1}^{g_{ij}-1}
\frac1{(1-\zeta^q)(1-\zeta^{-q})}
=\frac{g_{ij}^2-1}{12}.\]
where $\zeta$ is a primitive $g_{ij}$-root of unity.
Comparing the linear coefficients with (1) gives
\[\frac{S}{12P}
\left(S^2+\sum_{i<j}w_iw_j
+\sum_{i<j}(g_{ij}^2-1)w_kw_\ell\right)
=\frac{8d+24}{12}.\]
Using $P=S^3/(8d)$ gives 
\[\sum_{i<j}w_iw_j
+\sum_{i<j}(g_{ij}^2-1)w_kw_\ell=
\frac{3}{d}S^2.\]
Finally, by using $\sum_{i<j}(w_i-w_j)^2
=3S^2-8\sum_{i<j}w_iw_j$, we get
\[\sum_{i<j}(w_i-w_j)^2+3\left(\frac8d-1\right)S^2
=8\sum_{i<j}(g_{ij}^2-1)w_kw_\ell.
\tag{$\mathrm{HA}_{V_d}$}\]

Similarly, we can get the Hilbert identity for Fano threefolds with index one.
Let $Y_g$ be a Fano threefold with index one and fixed $g$.

The anticanonical Hilbert polynomial of $Y_g$ is
\[\chi(Y_g,-mK_{Y_g})=
\frac{2g-2}{6}m^3
+\frac{2g-2}{4}m^2
+\frac{g+11}{6}m+1.\]
In particular, $h^0(Y_g,-K_{Y_g})=g+2,$ so
\[ [t^S]\frac1{\prod_i(1-t^{w_i})}=g+2.\] 
Comparison of the linear anticanonical Hilbert coefficient gives
\[\sum_{i<j}w_iw_j+\sum_{i<j}(g_{ij}^2-1)w_kw_\ell=\frac{12}{g-1}S^2.\]
Equivalently,
\[\sum_{i<j}(w_i-w_j)^2+
\left(\frac{96}{g-1}-3\right)S^2
=8\sum_{i<j}(g_{ij}^2-1)w_kw_\ell.
\tag{$\mathrm{HA}_{Y_g}$}\]
\end{proof}

The following lemma is closely related to \cite[Corollary 2.21]{DLT}, but the hypotheses are different. DeVleming–Li–Torres's Corollary 2.21 starts only from their numerical equation for a degeneration of $\PP^n$, and deduces that every codimension-two singularity is $T$; when the stabilizer order is square-free, it is $A_{d-1}$. 
Here we assume the existence of an actual $\QQ$-Gorenstein smoothing, so the $T$-property follows directly and does not depend on the target's volume.

\begin{Lem}
Let $X=\PP(w_0,w_1,w_2,w_3)$ be well-formed and admit a $\QQ$-Gorenstein smoothing to a smooth threefold. Put
\[C_{ij}=\{x_k=x_\ell=0\}\simeq\PP(w_i,w_j),
\qquad d=g_{ij}:=\gcd(w_i,w_j)>1.\]
Then the transverse cyclic quotient singularity at the generic point of $C_{ij}$ is a $T$-singularity. Consequently, there exist positive integers $e,n,a$, with $d=en^2,\ \gcd(a,n)=1,$ such that the transverse singularity is
$\frac1{en^2}(1,ena-1)$. In particular, if $d$ is square-free, then $n=1$, and the transverse singularity is
$\frac1d(1,-1)=A_{d-1}.$ Thus $X$ has transverse $A_{d-1}$-singularities at the general point of $C_{ij}$.
\end{Lem}

\begin{proof} At the general point of $C_{ij}$, the stabilizer has order $d=\gcd(w_i,w_j)$. The two transverse coordinates are $x_k,x_\ell$, so the quotient surface is
$\frac1d(w_k,w_\ell)$. Well-formedness gives
$\gcd(w_k,d)=\gcd(w_\ell,d)=1$. Replacing the generator of $\mu_d$, we may therefore write the singularity as
\[\frac1d(1,q),
\qquad
q\equiv w_k^{-1}w_\ell\pmod d.\]
Now restrict the given smoothing locally to a general Cartier surface transverse to $C_{ij}$. If $\mathcal X\longrightarrow\Delta$ is the original $\QQ$-Gorenstein smoothing and $\mathcal S\subset\mathcal X$ is such a relative transverse Cartier divisor, then
\[K_{\mathcal S/\Delta} =\left(K_{\mathcal X/\Delta}+\mathcal S\right)|_{\mathcal S} \]
is $\QQ$-Cartier. Hence $\frac1d(1,q)$ admits a $\QQ$-Gorenstein smoothing. It is therefore a $T$-singularity.
The classification of cyclic $T$-singularities says that such a singularity has the form
\[\frac1{en^2}(1,ena-1),
\qquad \gcd(a,n)=1.\] 
Since its group order is $d$, we have $d=en^2$. If $d$ is square-free, then $n^2\mid d$ forces $n=1$. Therefore
\[q\equiv ea-1\equiv-1\pmod e,
\qquad e=d,\]
and hence $\frac1d(1,q)\cong\frac1d(1,-1)=A_{d-1}$.
\end{proof}

The following discussion is essentially contained in \cite[Proposition 4.8, Theorem 4.10]{DLT}. Let $X=\PP(w_0,w_1,w_2,w_3)$, let $S=\sum_iw_i,$ and $C_{ij}=\{x_k=x_\ell=0\}\simeq\PP(w_i,w_j), \ \{i,j,k,\ell\}=\{0,1,2,3\}.$ Let $g_{ij}=\gcd(w_i,w_j).$
The generic transverse singularity along $C_{ij}$ is $A_{g_{ij}-1}$ precisely when $g_{ij}\mid w_k+w_\ell.$
For $0\le m\le g_{ij}-2$, define $d_{i,m}$ and $d_{j,m}$ to be the smallest positive integers satisfying
\[w_i d_{i,m}+(m+1)(w_k+w_\ell)\equiv0\pmod {w_j},\]
\[w_j d_{j,m}+(m+1)(w_k+w_\ell)\equiv0\pmod {w_i}.\]
Set
\[b_{ij,m}=
\frac{w_i d_{i,m}+w_j d_{j,m}+(m+1)(w_k+w_\ell)}
{w_iw_j/g_{ij}}.\]
Then \(b_{ij,m}\) is a positive integer.
Define \(f_{ij}\) by
\[\left(\omega_X^{[-1]}|_{C_{ij}}
\right)_{\mathrm{torsion \ free \ part}}
\simeq\mathcal O_{\PP^1}(f_{ij}),\]
\cite[Proposition 4.8]{DLT} gives the integers \(b_{ij,m}\) and the splitting $Q_{ij}\cong\bigoplus_{m=0}^{g_{ij}-2}
\mathcal O_{\PP^1}(-b_{ij,m}),$ 
where \(Q\) is the cokernel of
$\Omega_X\longrightarrow\Omega_X^{[1]}$, and $Q_{ij}=(Q|_{C_{ij}})/{\rm Tors}(Q|_{C_{ij}})$.
The proof of \cite[Theorem 4.10]{DLT} then produces, up to sheaves supported at the coordinate points, summands of degree
$-2+b_{ij,m}+f_{ij},$ and proves
$H^1\!\left(X,\mathcal Ext^1(\Omega_X,\mathcal O_X)\right)=0$
by showing
\[H^1(X,\mathcal Ext^1(\Omega_X,\mathcal O_X))
\simeq  \oplus_{i, j}H^1(C_{ij}, \mathcal Ext^2(Q_{ij}, \omega_X)\otimes \omega_X^\vee)=0\]
when every singular coordinate curve is of $A$-type and an anticanonical section does not vanish generically along each such curve.

Since $\omega_X^{[-1]}=\mathcal O_X(S)$, $f_{ij}$ can be computed directly from the restriction of degree-$S$ monomials to $C_{ij}$.
Up to zero-dimensional modifications at the two coordinate endpoints, a direct sufficient condition is
\[b_{ij,m}+f_{ij}\ge1 \
\text{for every singular curve }C_{ij}
\text{ and every }0\le m\le g_{ij}-2.\]
Because $b_{ij,m}\ge1$, it is enough to know $f_{ij}\ge0.$ 

Let $a$ and $b$ be two positive integers with $\gcd(a,b)=1$. The Frobenius theorem says that the largest integer not expressible as
$\alpha a+\beta b,\ \alpha,\beta\in\mathbb Z_{\ge0},$ is
$ab-a-b.$ Consequently,
\[N>ab-a-b\quad\Longrightarrow\quad
N\in\langle a,b\rangle_{\mathbb Z_{\ge0}}.\]

This is guaranteed by the existence of an anticanonical monomial which does not vanish generically on $C_{ij}$. It is equivalent to, $S\in\left\langle w_i,w_j\right\rangle_{\mathbb Z\ge0}$. 
Then $x_i^\alpha x_j^\beta\in H^0(X,\mathcal O_X(S))$ restricts nontrivially to $C_{ij}$, so $f_{ij}\ge0$.

Thus, if $X$ has three $A$-type coordinate curves and condition $S\in\left\langle w_i,w_j\right\rangle_{\mathbb Z\ge0}$ holds for all three, then
$H^1\!\left(X,\mathcal Ext^1(\Omega_X,\mathcal O_X)\right)=0$
The intersections of the three curves occur only at finitely many coordinate points. They do not contribute $H^1$, but they affect compatibility and local obstruction questions at the codimension-three points.

\begin{Lem}\label{vanishing}
If $X=\PP(w_0,w_1,w_2,w_3)$ has all singular coordinate curves $A$-type coordinate curves and condition 
\[S=w_0+w_1+w_2+w_3\in\langle w_i,w_j\rangle_{\mathbb Z_{\ge0}}\]
holds for all three coordinate curves, then
$H^1\!\left(X,\mathcal Ext^1(\Omega_X,\mathcal O_X)\right)=0.$
\end{Lem}

\begin{Lem}\label{T^1}
Suppose $C\subset X$ is a curve with transversal $A_{r-1}$-singularities. On the canonical cover, let
$N_{C/X}=L\oplus M,$ where $\mu_r$ acts on $L,M$ with opposite characters. Then $\mathcal T^1_{\mathrm{QG}}|_C
\simeq\bigoplus_{j=0}^{r-2}
(\det N_{C/X})^{\otimes(r-j)}.$
\end{Lem}

\begin{proof}
The invariant transverse equation is
$uv=w^r,$ and its first-order deformations are
\[uv=w^r+s_{r-2}w^{r-2}+\cdots+s_1w+s_0.\]
The coefficient $s_j$ transforms as a section of
$(L\otimes M)^{\otimes(r-j)}.$ Therefore,
\[\mathcal T^1_{\mathrm{QG}}|_C
\simeq\bigoplus_{j=0}^{r-2}
(\det N_{C/X})^{\otimes(r-j)}.\]
\end{proof}

\begin{Prop}
Let $Y_0=\{uv=z^{n+1}\}\subset \mathbb A^4_{u,v,z,q}$, where the $q$-axis is the singular locus and the generic transverse singularity is $A_n$.
Suppose $G=\mu_m$ generated by $\xi$ acts diagonally by
$\xi(u,v,z,q)=\bigl(\zeta_m^a u,\zeta_m^bv,\zeta_m^cz,\zeta_m^d q\bigr)$, where $\zeta_m$ is a primitive $m$-th root of unity.  We set $e\in\mathbb Z/m\mathbb Z \ \text{by}\ 
e\equiv a+b\equiv(n+1)c\pmod m.$

Assume also that
\begin{itemize}
\item $G$ acts freely on $Y\setminus\{0\}$; 
\item $e\not\equiv0,\ c,\ d\pmod m$.
\end{itemize}
Then $X_0:=Y_0/G$ has no smoothing compatible with a smoothing of the transverse $A_n$-singularities. In fact, every deformation of $X_0$ remains singular at the endpoint.
\end{Prop}

\begin{proof}
Every deformation of $Y_0$ can be written
\[uv=z^{n+1}
+\sum_{j=0}^{n-1}f_j(q,\mathbf t)z^j.\tag{1}\]
Because $G$ acts trivially on the deformation base, an equivariant monomial $q^kz^j$ can occur in (1) only when
$jc+kd \equiv e \pmod m.$ There are exactly three kinds of terms of degree at most one:
$1,\ z,\  q$ of weights $0, c, d$ respectively.  

By the assumption, none of the above terms $1, z, q$ is allowed. Therefore, every perturbing term belongs to
$(z,q)^2.$ Thus every equivariant deformation $\mathcal{X} \to T$ is given by  the form
\[uv-z^{n+1}-H(z,q,\mathbf t)=0,\qquad H\in(z,q)^2.\]
The origin remains on every fibre, and all partial derivatives vanish. Therefore every fibre of the covering deformation remains singular at the origin.
Because $Y_0\setminus\{0\}\to X_0\setminus\{0\}$ is finite \'etale, the standard Schlessinger cover-lifting argument \cite{Sch} shows that every formal deformation of $X_0$ lifts to a $G$-equivariant deformation of $Y_0$.

Suppose now that some quotient fiber $X_t$ were smooth, then
$Y_t\longrightarrow X_t$ would be finite \'etale away from the origin. Hence the branch locus consists of at most one point. Purity of the branch locus asserts that the branch locus is purely of codimension 1 since $X_t$ is smooth. Therefore, $Y_t\to X_t$ is \'etale everywhere. However, the origin is fixed by $G$, so the quotient map is not étale there. This is the required contradiction.
\end{proof}

\begin{Cor}\label{non-smoothing}
Let $X=\PP(w_0, w_1, w_2, w_3)$.
Consider an endpoint
$P_j\in C_{ij}\subset\PP(w_i,w_j,w_k,w_\ell)$ and put
\[g:=\gcd(w_i,w_j)=n+1,\qquad m:=\frac {w_j} g.\]
Assume the transverse singularity is $A_{g-1}$, so
$g\mid w_k+w_\ell$.

After quotienting by the stabilizer $\mu_g$, the index one cover has equation $uv=z^g$ and the $\mu_m$-weights are
\[a \equiv w_k,\qquad
b \equiv w_\ell,\qquad
c \equiv\frac{w_k+w_\ell}{g},\qquad
d \equiv\frac{w_i}{g}\pmod m.\]
Consequently, provided the action is free away from the endpoint, there is no compatible endpoint smoothing whenever $ e:\equiv a+b \not \equiv 0, c, d \pmod m$.
\end{Cor}

\begin{Ex}\label{P(2)}
Let $X=\PP(2,5,8,25)$. This is a numerical candidate for $V_4$. Consider the endpoint $P_2$ in the singular curve $C_{02}$, which is the $(2, 8)$-coordinate curve. We have 
$w_j=8,\ g=2,\ m=4$. The transverse weights are $1,5$, so $a \equiv 1$, $b \equiv 1$, $c \equiv 3$, $d \equiv 1$ and $e \equiv 2 \pmod 4$. 
By the criterion in Corollary~\ref{non-smoothing}, $X$ cannot be smoothed.
\end{Ex}

Except Example~\ref{P(2)}, we cannot exclude any other numerical candidates in Table (1) in Introduction by Corollary~\ref{non-smoothing}.

\section{Preserving Picard rank via \(\QQ\)-Gorenstein smoothing}

Let $X=\PP(w_0,w_1,w_2,w_3)$ be a $\QQ$-Gorenstein degeneration of a smooth Fano threefold $Y$. Then one need not have $\rho(Y)=1$. The techniques developed in this section are used in proving Proposition~\ref{Y6}, Proposition~\ref{Y10}, and Proposition~\ref{Y12}.

\begin{Rmk}\label{jumping}
The explicit example $X=\PP(1,2,3,6)$  has two different $\QQ$-Gorenstein smoothing components, whose general fibres have Picard numbers 2 and 3.

Let $S=\PP(1,2,3).$ The complete linear system $|\mathcal O_S(6)|=|-K_S|$ embeds $\PP(1,2,3)$ as a weak Pezzo surface of degree six in $\PP^6$ with one $A_1$- and one $A_2$-singularity. Its projective cone is
\[\operatorname{Cone}(S)\simeq\PP(1,2,3,6)
\subset\PP^7,\]embedded by
$H=\mathcal O_X(6)$. Moreover,
\[-K_X=\mathcal O_X(12)=2H,
\qquad H^3=\frac{6^3}{1\cdot2\cdot3\cdot6}=6,\]
so $(-K_X)^3=8H^3=48$. The two smoothing components are described by Brown-Reid-Stevens using the Tom and Jerry unprojection formats \cite{BRS}.

One component smooths $X$ to a general hyperplane section
$Y_2\in\left|\mathcal O_{\PP^2\times\PP^2}(1,1)\right|.$ By the Lefschetz theorem,
$\operatorname{Pic}(Y_2)\simeq
\operatorname{Pic}(\PP^2\times\PP^2)
\simeq\mathbb Z^2,$ and hence
$\rho(Y_2)=2.$ Adjunction gives
$-K_{Y_2}=\mathcal O_{Y_2}(2,2)=2\mathcal O_{Y_2}(1,1)$.

The other component smooths $X$ to
\[Y_3=\PP^1\times\PP^1\times\PP^1
\subset\PP^7\]
in its Segre embedding. Here
$\operatorname{Pic}(Y_3)\simeq\mathbb Z^3, \ 
\rho(Y_3)=3,$ and
$-K_{Y_3}=\mathcal O_{Y_3}(2,2,2)=2\mathcal O_{Y_3}(1,1,1).$
Both smooth fibres  satisfy
$(-K)^3=48,$ as required by deformation invariance. The unprojection families are Gorenstein, so in particular they are $\QQ$-Gorenstein. 
\end{Rmk}

To check the Picard rank of a general fiber of $\QQ$-Gorenstein smoothing, we use the nearby- and vanishing-cycles.
For their construction, the specialization triangle, proper base change, and the identification of their stalks with Milnor fibre cohomology,
see \cite[Section 3]{Max20} or \cite[\S4.2]{Dimca}. For the Picard--Lefschetz monodromy formula, see \cite{Del73}.

Let $\pi:\mathcal X\longrightarrow \Delta$ be the projective smoothing with central fibre $X$ and smooth general fibre $Y$. There is a distinguished triangle, 
\[\QQ_X\longrightarrow R\psi_\pi\QQ
\longrightarrow R\phi_\pi\QQ
\xrightarrow{+1}\]
where $R\psi_\pi\QQ$ is the nearby cycle complex, and   $R\phi_\pi\QQ$ is the vanishing-cycle complex. 

Proper base change identifies
\[\mathbb H^k(X,R\psi_\pi\QQ)\simeq H^k(Y,\QQ).\]
Thus $R\phi_\pi\QQ$ measures precisely the topological difference between $X$ and $Y$. Taking hypercohomology gives
\[H^2(X,\QQ)\xrightarrow{\mathrm{sp}}H^2(Y,\QQ)
\longrightarrow\mathbb H^2(X,R\phi_\pi\QQ)
\longrightarrow H^3(X,\QQ).\]
Exactness says
\[\operatorname{coker}\!\left(H^2(X)\to H^2(Y)\right)=
\ker\!\left(\mathbb H^2(X,R\phi_\pi\QQ)\to H^3(X)
\right).\]
Therefore, if
$\mathbb H^2(X,R\phi_\pi\QQ)=0,$ then every degree-two cohomology class on $Y$ comes from $X$:
\[H^2(X,\QQ)\twoheadrightarrow H^2(Y,\QQ).\]

A weighted projective threefold $X=\PP(w_0,w_1,w_2,w_3)$ has the same rational cohomology as $\PP^3$:
\[H^k\!\left(\PP(w_0,w_1,w_2,w_3),\QQ\right)=
\begin{cases}
\QQ,&k=0,2,4,6,\\
0,&k\text{ odd}.
\end{cases}\]
In particular,
\[H^2(X,\QQ)=\QQ,\qquad H^3(X,\QQ)=0.\]
The generator of $H^2(X,\QQ)$ is the rational hyperplane class $c_1(\mathcal O_X(1))$. If \(\mathbb H^2(X,R\phi_\pi\QQ)=0\), specialization gives
\[\QQ=H^2(X,\QQ)\twoheadrightarrow H^2(Y,\QQ).\] Since $Y$ is projective manifold, $H^2(Y,\QQ)\neq0$. Hence
\[H^2(Y,\QQ)\simeq\QQ,\qquad b_2(Y)=1.\]
For a smooth Fano threefold,
$H^1(Y,\mathcal O_Y)=H^2(Y,\mathcal O_Y)=0$ by Kodaira vanishing theorem. The exponential sequence then identifies divisor classes with integral degree-two cohomology, up to irrelevant torsion. Consequently, $\rho(Y)=b_2(Y)=1$.

The family is locally topologically trivial on the smooth locus of $X$. Therefore $R\phi_\pi\QQ$ is supported on the singular curves and their special endpoints.
At a general point of an $A_{r-1}$-curve, the transverse central equation is $uv=z^r.$ Its Milnor fiber $F$  has the homotopy type of a bouquet of $r-1$ two-spheres:
\[\widetilde H^q(F,\QQ)=
\begin{cases}
\QQ^{\,r-1},&q=2,\\
0,&q\neq2.
\end{cases}\]
The degree-two group is naturally the root lattice
$A_{r-1}\otimes\QQ.$  As the point moves along the singular curve, these groups form a local system $\mathbb V$. Its monodromy comes from permuting the roots of the deformed polynomial, equivalently from the Weyl group of $A_{r-1}$.
The hypercohomology spectral sequence is
\[E_2^{p,q}=H^p(X,R^q\phi_\pi\QQ)
\Longrightarrow
\mathbb H^{p+q}(X,R\phi_\pi\QQ).\]
Since $R^q\phi_\pi\QQ=0$ for $q<2$, the only contribution to total degree two is
\[\mathbb H^2(X,R\phi_\pi\QQ)=
H^0(X,R^2\phi_\pi\QQ).\]
But $H^0$ of a local system is its monodromy-invariant part.
So the relevant group is not the direct sum of all local Milnor lattices. It consists only of vanishing-cycle classes that can be transported consistently around the entire singular curve and through its endpoints.

\begin{Lem}\label{rho1}
Let $\pi:\mathcal X\longrightarrow \Delta$ be the projective smoothing with central fibre $X$ and smooth Fano threefold $Y$. 
If $H^2(X, \QQ)\simeq\QQ$ and $\mathbb H^2(X,R\phi_\pi\QQ)=0$ then $\rho(Y)=1$.
\end{Lem}

\begin{Rmk}\label{jumping-2}
In Remark~\ref{jumping}, the vanishing cycle term is nonzero. That is exactly why the Picard number increases.

Let
$X=\PP(1,2,3,6),\ H=\mathcal O_X(6).$ Then
\[H^2(X,\QQ)=\QQ[H],
\qquad H^3(X,\QQ)=0.\]For either smoothing component, the nearby-cycle sequence becomes
\[H^2(X,\QQ)\longrightarrow H^2(Y,\QQ)
\longrightarrow\mathbb H^2(X,R\phi_\pi\QQ)
\longrightarrow0.\]The class $H$ extends as the projective hyperplane class, so the first map is nonzero and therefore injective. Hence
\[\mathbb H^2(X,R\phi_\pi\QQ)\simeq
\frac{H^2(Y,\QQ)}{\QQ[H_Y]}.\]For the first smoothing, $Y_2\subset \PP^2\times(\PP^2)^\vee.$ Write $h_1,h_2$ for the two ambient hyperplane classes. Then
\[H^2(Y_2,\QQ)
=\QQ h_1\oplus\QQ h_2,\]and the central hyperplane class specializes to
$H\longmapsto h_1+h_2.$
Therefore
\[\mathbb H^2(X,R\phi\QQ)\simeq
\frac{\QQ h_1\oplus\QQ h_2}{\QQ(h_1+h_2)}
\simeq\QQ.\]
The additional class can be represented by $h_1-h_2$. Thus
$b_2(Y_2)=2$.

For the Segre smoothing,
$Y_{3}=(\PP^1)^3$. Writing $h_1,h_2,h_3$ for the three factor classes,
\[H^2(Y_3,\QQ)
=\QQ h_1\oplus\QQ h_2\oplus\QQ h_3,\]
while $H\longmapsto h_1+h_2+h_3.$

Hence
\[\mathbb H^2(X,R\phi\QQ)\simeq
\frac{\QQ^3}{\QQ(1,1,1)}\simeq\QQ^2.\] Thus $b_2(Y_3)=3.$
\end{Rmk}

\begin{Lem}\label{monodromy}
Let $\pi:\mathcal Z\longrightarrow\Delta$ be a projective $\QQ$-Gorenstein smoothing of a projective threefold $Z$, with a smooth general fiber $Y$. 
Let $C\subset Z$ be a singular curve with generic transverse equation $uv=z^r.$ Suppose that the smoothing along $C$ has the form $uv=p_q(z)$, where
$p_q(z)=z^r+s_{r-2}(q)z^{r-2}+\cdots+s_0(q).$
Let $C^\circ$ be the complement of the discriminant and endpoint points on an $A_{r-1}$-curve. 
Let $G\subset S_r$ be the monodromy group of the roots of $p_q(z)$. 

Suppose $p_q(z)$ has $r$ distinct roots for $q\in C^\circ$, and $G$ acts transitively on the $r$ roots.
Then the vanishing-cycle local system $\mathbb V$ along $C^\circ$ has no nonzero global section:
\[A_{r-1,\QQ}^{G}\simeq H^0(C^\circ,\mathbb V)=0.\]
In particular, this holds if $G=S_r$, or if $G$ contains an $r$-cycle.
\end{Lem}

\begin{proof}
Fix a base point $q_0\in C^\circ$ and label the roots $\lambda_1,\ldots,\lambda_r.$ The transverse Milnor fiber is
$F_{q_0}=\{uv=p_{q_0}(z)\},$ and its degree-two reduced cohomology is the rational $A_{r-1}$-root lattice:
\[ H^2(F_{q_0},\QQ)\simeq A_{r-1,\QQ}=
\left\{(a_1,\ldots ,a_r)\in\QQ^r:
\sum_{i=1}^r a_i=0\right\}.\]
The roots $e_i-e_j$ correspond to the vanishing cycles associated with paths joining $\lambda_i$ and $\lambda_j$.
Analytic continuation around a loop in $C^\circ$ permutes the roots $\lambda_i$. Under the above identification, the resulting monodromy on the vanishing-cycle lattice is the corresponding permutation of the coordinates:
\[\sigma(a_1,\ldots ,a_r)=
(a_{\sigma^{-1}(1)},\ldots ,a_{\sigma^{-1}(r)}).\]
Thus the monodromy representation on $A_{r-1,\QQ}$ is the restriction of the permutation representation of $G\subset S_r$.
Now let
\[a=(a_1,\ldots ,a_r)\in A_{r-1,\QQ}^{G}.\]
Because $G$ is transitive, for every $i,j$ there exists $\sigma\in G$ such that $\sigma(i)=j$. Since $a$ is $G$-invariant, $a_i=a_j.$ Therefore all coordinates are equal:
$a_1=\cdots=a_r=\lambda.$ But $a\in A_{r-1,\QQ}$, so
$0=\sum_{i=1}^r a_i=r\lambda.$ Over $\QQ$, this implies $\lambda=0$. Hence
$A_{r-1,\QQ}^{G}=0$. Finally, global sections of a local system are precisely monodromy-invariant vectors, so 
$H^0(C^\circ,\mathbb V)\simeq A_{r-1,\QQ}^{G}=0$.
\end{proof}

\section{Constructing examples}\label{theorem}

\begin{Prop}
The weighted projective space $X=\PP(5,12,18,25)$ admits a $\QQ$-Gorenstein smoothing to $V_1$, which is a degree $6$ hypersurface $X_6 \subset\PP(1,1,1,2,3).$
\end{Prop}

\begin{proof}
Let
\[X=\PP(5,12,18,25)
=\operatorname{Proj}\mathbb C[x_0,x_1,x_2,x_3],\]where $\deg(x_0,x_1,x_2,x_3)=(5,12,18,25)$.

Consider the following monomials:
\[\begin{array}{c|c|c}
\text{variable}&\text{monomial}&\text{degree divided by }30\\ \hline
y_0&x_1x_2&1\\
y_1&x_0x_3&1\\
y_2&x_0^6&1\\
z&x_1^5&2\\
u&x_2^5&3\\
v&x_3^6&5
\end{array}\]
Thus these give weighted coordinates of weights
$(1,1,1,2,3,5).$ They satisfy
$y_0^5=zu,\ y_1^6=y_2v,$ and
\[X\simeq
\{y_0^5=zu, y_1^6=y_2v \}
\subset\PP(1,1,1,2,3,5). \]This is a weighted complete intersection of degrees 5 and 6.
Let
$g_6(y_0,y_1,y_2,z,u)$ be a sufficiently general weighted homogeneous polynomial of degree 6.
Define
$\mathcal X\subset\PP(1,1,1,2,3,5)\times\mathbb A^1_t$ by
\[\left\{\begin{array}{l}
y_0^5-zu=tv,\\
y_1^6-y_2v=tg_6.
\end{array}\right.\]

At $t=0$, one sees that  $\mathcal X_0 =X \simeq\PP(5,12,18,25)$.
For $t\neq0$,  we obtain
\[\mathcal{X}_t=\{ty_1^6-y_2y_0^5+y_2zu-t^2g_6=0 \} \subset \mathbb{P}(1,1,1,2,3). \] Therefore, for $t\neq0$, $\mathcal X_t\simeq
X_{6}\subset\PP(1,1,1,2,3)$ for a general choice of $g_6$, which is of type $V_1$. 
\end{proof}

\begin{Prop}
The weighted projective space $X=\PP(3,4,8,9)$ admits a $\QQ$-Gorenstein smoothing to $V_2.$
\end{Prop}

\begin{proof}
Let $X=\PP(3,4,8,9)= \operatorname{Proj}\mathbb C[x_0,x_1,x_2,x_3]$, where $\deg(x_0,x_1,x_2,x_3)=(3,4,8,9)$. The nontrivial pairwise gcds are
$\gcd(3,9)=3,\ \gcd(4,8)=4.$ Thus the singular locus consists of two disjoint curves
\[C_3=\{x_1=x_2=0\}\simeq\PP(3,9)\simeq\PP(1,3),
\qquad C_4=\{x_0=x_3=0\}\simeq\PP(4,8)\simeq\PP(1,2).\]

Along $C_3$,  the generic transverse singularity is $A_2$, and along $C_4$, the generic transverse singularity is $A_3$.
For $C_3$, take $s_3=x_0^3+x_3\in H^0\bigl(\PP(1,3),\mathcal O(3)\bigr)$. This is a smoothing coefficient for the transverse $A_2$-equation $uv=w^3+\tau s_3$. Its zero is simple because
$\frac{\partial (x_0^3+x_3)}{\partial x_3}=1.$
For $C_4$, take
$s_4=x_1^2+x_2\in H^0\bigl(\PP(1,2),\mathcal O(2)\bigr).$ It gives $uv=w^4+\tau s_4$, and again its zero is simple:
$\frac{\partial(x_1^2+x_2)}{\partial x_2}=1$. Hence these equations smooth the generic points of both singular curves, including the points at which $s_3$ or $s_4$ vanishes.

The nontrivial endpoint of $C_4$ is
$P_2=[0:0:1:0]$, whose  local singularity is of type $\frac18(3,4,1)$. Take $G_4$ the order-four subgroup of $\mu _8$ fixing the coordinate of weight $4$. We consider 
\[u=x_0^4,\qquad v=x_0 x_3,\qquad w=x_3^4\] satisfying with relation
$uw=v^4.$ The residual \(\mu _2\) acts by
$(u,v,w,x_1)\longmapsto(-u,-v,-w,-x_1)$. Therefore
\[\frac18(3,4,1)\simeq\{uw=v^4\}/\mu _2.\]
On the chart $x_2=1$, the section $s_4=x_1^2+x_2$ becomes $x_1^2+1$. Consider
$\mathcal U_2:
uw=v^4+\tau(x_1^2+1)$ with the same \(\mu _2\)-action. 
For \(\tau\neq0\), the hypersurface upstairs is smooth. 
Moreover, the only fixed point of $\mu_2$ in the ambient affine space is the origin, which is absent from the fiber for $\tau\neq0$. Hence, the quotient fiber is smooth.

Similarly, the non-trivial endpoint of $C_3$ is $P_3=[0:0:0:1]$, which is of type  $\frac19(3,4,8)$.
Take $G_3$  the order three subgroup of $\mu_9$ fixing  the coordinate of weight $3$ acts on the other two transverse coordinates with weights $1,2$. We consider 
\[u=x_1^3,\qquad v=x_1x_2,\qquad w=x_2^3\]
satisfying $uw=v^3.$ The residual  $\mu_3$ acts by $(u,v,w,x_0) \mapsto (\zeta_3u, \zeta_3 v, \zeta_3^2 w, \zeta_3 x_0)$ with weights $(1,1,2,1)\pmod3.$ Thus
\[\frac19(3,4,8)\simeq\{uw=v^3\}/\mu _3.\]
On the chart $x_3=1$, $s_3=x_0^3+x_3$ becomes $x_0^3+1$. Consider
$\mathcal U_3: uw=v^3+\tau(x_0^3+1).$ For $\tau\neq0$, its index-one cover is smooth. The $\mu_3$-action has no fixed point on this fiber because all four coordinates have nonzero $\mu _3$-weights and the origin is absent. Thus, the quotient is smooth.
These are $\QQ$-Gorenstein deformations because they are equivariant deformations of the corresponding index-one covers.

On the two charts covering $C_3$, the smoothing coefficient is obtained from the single homogeneous section $x_0^3+x_3$. On the charts covering $C_4$, it comes from $x_1^2+x_2$. Therefore the local deformations agree on chart overlaps. 

It remains to check the vanishing hypothesis in the local-to-global theorem. We have
$-K_X=\mathcal O_X(3+4+8+9)=\mathcal O_X(24)$. Then Lemma~\ref{vanishing} implies
$H^1\!\left(X,\mathcal Ext^1(\Omega_X,\mathcal O_X)\right)=0$, because $24=3\cdot 8=4\cdot 6$. Hence, the compatible local smoothings glue to a global \(\QQ\)-Gorenstein smoothing.

Let $Y$ be a smooth general fiber. Since the deformation is $\QQ$-Gorenstein,
\[(-K_Y)^3=(-K_X)^3=\frac{24^3}{3\cdot4\cdot8\cdot9}=16.\]
Let $\mathcal L_0:=\mathcal O_X(12)$. Then $\mathcal L_0^{[2]}\simeq\mathcal O_X(24)\simeq\omega_X^{-1}$.
We claim that $\mathcal L_0$ extends through the smoothing.
Along the two singular curves, the generic stabilizers have orders $3$ and $4$. Since
$3| 12,\ 4| 12,$ the sheaf $\mathcal O_X(12)$ is locally Cartier at their generic points.
The only issue is at the endpoints of weights $8$ and $9$.
At the weight-$8$ point,
$\frac18(3,4,1)\simeq\{uw=v^4\}/\mu _2$, where $\mu _2$ acts by $-1$ on $u,v,w,x_1$. The class $12\pmod8$ equals $4$, so $\mathcal O_X(12)$ corresponds precisely to the nontrivial character of this residual $\mu_2$. In the smoothing
\[\mathcal U_8: uw=v^4+\tau(x_1^2+1),\]the same character defines a rank-one sheaf on the entire quotient family $\mathcal U_8/\mu_2$ which is $\QQ$-Gorenstein. Thus $\mathcal O_X(12)$ extends locally. At the weight-$9$ point,
$\frac19(3,4,8)\simeq\{uw=v^3\}/\mu _3$, with residual weights $(1,1,2,1)$. Since $12\equiv3\pmod9$, $\mathcal O_X(12)$ gives a character of the residual $\mu _3$, which extends across
\[uw=v^3+\tau(x_0^3+1).\]
Therefore $\mathcal O_X(12)$ extends over all the local smoothing charts.
Weighted projective space satisfies
\[H^1(X,\mathcal O_X)=H^2(X,\mathcal O_X)=0.\]
The obstruction to gluing the local extensions of a line bundle lies in $H^2(\mathcal O_X)$, while ambiguity lies in $H^1(\mathcal O_X)$. Thus the local extensions glue uniquely to a reflexive sheaf $\mathcal L$ on the global smoothing
$\pi:\mathcal X\longrightarrow\Delta$.

Let $Y=\mathcal X_t$ for $t\neq0$, and set $L=\mathcal L|_Y$. Because $Y$ is smooth, $L$ is a line bundle, and $-K_Y=2L.$ Moreover, $(-K_Y)^3=16,$ so $L^3=2.$ By the classification of Fano threefolds,  $Y\simeq V_2$ because $ Y$ has Fano index 2 and $(-K_Y)^3=16.$
\end{proof}

\begin{Prop}
The weighted projective space $X=\PP(3,14,18,49)$ has a $\QQ$-Gorenstein smoothing to $V_2$.
\end{Prop}

\begin{proof}
We have the following gcd graph
\[3\xleftrightarrow{\ 3\ }18
\xleftrightarrow{\ 2\ }14
\xleftrightarrow{\ 7\ }49.\]
The transversal singularities are of type 
$A_2,  A_1$ and $A_6$ respectively. Moreover,
$-K_X=\mathcal O_X(84)$ and $(-K_X)^3=16.$

We set coordinate $x_0,x_1,x_2,x_3$ with weights $3,14,18,49$ respectively. We first consider an embedding $X \hookrightarrow \mathbb{P}(1,6,14,21,49)_{x_0,x_2,u,v,w}$ by introducing $u=x_1^3, v=x_1x_3, w=x_3^3$ and dividing all the weights by $3$. This leads the following partial smoothing

\[\mathcal X=\{uw-v^3-t\bigl(x_0^{63}+x_2^7v\bigr)=0\}
\subset\PP(1,6,14,21,49)\times\mathbb A^1_t.\]

Clearly, the central fiber is $X$. For $t \ne 0$, one sees that the  general fiber
\[X_{63}=\{uw=v^3+x_0^{63}+x_2^7v\} \subset \mathbb{P}(1,6,14,21,49)\]
is quasismooth. Thus $\mathcal{X}$ completely smooths the 1-dimenaional $A_2$ singularities along the curve $C_3=\{x_0=x_2=0\}$.
Also, $K_{X_{63}}=\mathcal O_{X_{63}}(-28),$ so this is a $\QQ$-Gorenstein partial smoothing.

The remaining $A_1$- and $A_6$-curves meet at the original $x_1$-coordinate point, which is now the $u$-coordinate point, denotes $P_u$. More precisely: 
\begin{itemize}
    \item The original $A_1$-curve $C_2=\{x_0=x_3=0\}$ goes to $\{x_0=v=w=0\}\subset X_{63}$ with two end points $(P_u\in X_{63})\simeq\frac{1}{14}(1,6,7)$ and $(P_{x_2}\in X_{63})\simeq\frac{1}{6}(1,2,1)$.
    \item The original $A_6$-curve $C_7=\{x_0=x_2=0\}$ goes to $\{x_0=x_2=0,\ uw=v^3\}\subset X_{63}$ with two end points $(P_u\in X_{63})\simeq\frac{1}{14}(1,6,7)$ and $(P_w\in X_{63})\simeq\frac{1}{49}(1,6,21)$. 
\end{itemize}

The germ $(P_u\in X_{63})\simeq\frac1{14}(1,6,7)$ is locally $\QQ$-Gorenstein smoothable. To see this, first quotient by the subgroup $G_7 \subset\mu _{14}$  of order $7$. The index-one presentation is $\{uv=w^7\}\big/\mu _2,$
where the $\mu _2$-weights are
$(u,v,w,z)=(1,0,1,1).$ The equivariant deformation
$uv=w^7+sz$ is invariant as a zero locus. For $s\ne0$, eliminate $Z$; the remaining germ is
$\frac12(1,0,1)\simeq\mathbb A^1\times A_1.$ The remaining $A_1$ has its standard smoothing. Thus this endpoint has a compatible smoothing of its $A_1$- and $A_6$-branches.
Consider next the germ  is
$P_w \in X_{63} \cong \frac1{49}(1,6,21).$ Its index-one cover is
$\mathbb A^1\times A_6 \cong \{uv=w^7\},$
with the $\mu _7$-weights
$(u,v,w,z)=(1,6,1,3).$ The deformation
$uv=w^7+\tau$ has smooth cover for $\tau\ne0$, and the action is free. Hence this endpoint is also $\QQ$-Gorenstein smoothable.

Notice that at the $y$-coordinate point $P_{x_2} \in X_{63}\simeq\frac16(1,2,1)$ is not suitable for completing the smoothing in a second step. We therefore return to $X$ and the original germ $(P_{x_2} \in X)=\frac1{18}(3,14,13)$ and construct a simultaneous smoothing of its intersecting $A_2$ and $A_1$-branches. At $P_{x_2}=[0:0:1:0]$, the germ is
$(P_y\in X)
\simeq\frac1{18}(3,14,13).$ Its canonical character is $30\equiv12\pmod{18}$,
so its canonical index is $3$. The index-one cover is
$\widetilde Q=\frac16(3,2,1)$, which is 
 the anticanonical affine cone over $\PP(3,2,1)$. The residual group is $\mu _3$, and
$(P_{x_2} \in X)=\widetilde Q/\mu _3.$
Write the coordinates on the original \(\mathbb A^3\) defining \(\widetilde Q\) as \(u,v,w\), of \(\mu _6\)-weights \(1,2,3\). The anticanonical invariant coordinates used by Brown-Reid-Stevens \cite{BRS} are
\[\alpha=u^6,\quad \beta=u^4v,\quad \gamma=u^2v^2,\quad
\xi=v^3,\quad \delta=u^3w,\quad \epsilon=uvw,\quad \eta=w^2.
\]The residual $\mu _3$-weights are
$\operatorname{wt}_{\mu _3}(\alpha,\beta,\gamma,\xi,\delta,\epsilon,\eta)
=(1,2,0,1,1,2,1).$ Brown-Reid-Stevens \cite{BRS} describe the Jerry smoothing component of $\widetilde Q$ by introducing deformation variables $\theta,\mu,\nu,\ldots$. The above residual action extends to that family with
\[\operatorname{wt}_{\mu _3}(\theta)=
\operatorname{wt}_{\mu _3}(\mu)=0,
\qquad
\operatorname{wt}_{\mu _3}(\nu)=2.\]
This can be checked directly from their Jerry matrix: the entry $x_1+\mu \eta$ has weight 1, while $\theta+\lambda\xi$ has weight 0.
Take the one-parameter small curve in the Jerry component
\[\theta=s,\qquad \mu=s,\qquad
\nu=\lambda=x_2=x_1=0.\]Because $\theta$ and $\mu$ have weight zero, this is a fibrewise $\mu _3$-equivariant deformation over the ordinary base $\mathbb A^1_s$.
For $s\ne0$, the relevant Jerry cubic is, up to signs, $\Phi(\epsilon,\eta)=\epsilon^3-s\eta^3.$ It has three distinct roots. Hence the general fiber of the index-one-cover deformation is smooth by the Jerry construction in Brown-Reid-Stevens \cite{BRS}

It remains to show that the $\mu _3$-action is free. A point fixed by a nonidentity element must satisfy
\[\alpha=\beta=\xi=
\delta=\epsilon=\eta=0,\] with only the weight-zero coordinate $\gamma$ possibly nonzero. Among the Jerry equations are, up to sign conventions,
$\beta \eta+ \gamma \theta-\delta \epsilon=0, \alpha \epsilon-\beta \delta-\mu \theta^2=0$. Since $\theta=s\ne0$, the first gives $\gamma=0$. The second then gives $0=\mu \theta^2=s^3$, a contradiction. Thus, the residual $\mu_3$-action is free on the general fibre.

Consequently, $\widetilde{\mathcal Q}/\mu _3\longrightarrow\mathbb A^1_s$ is a local $\QQ$-Gorenstein smoothing of $\frac1{18}(3,14,13)$, compatible with the $A_2$ and $A_1$-branches through $P_y$. This shows that the $A_2$-curve can be smoothed compatibly.

The smoothing-section conditions (S) hold: $52=5\cdot9+1\cdot7$ for the middle $A_1$-curve, and
$21=3\cdot7$ for the $A_6$-curve, so there are anticanonical sections nonzero generically along all three singular curves. The explicit family supplies compatibility along the original $A_2$-curve, while the above argument gives compatibility where the remaining $A_1$ and $A_6$-curves meet. 
$-K_X=\mathcal O_X(84)$, and the singular curves have transverse types $A_2,A_1,A_6$. Moreover,
\[84\in\langle3,18\rangle,\qquad
84\in\langle18,14\rangle,\qquad
84\in\langle14,49\rangle,\]
then by Lemma~\ref{vanishing} the local-to-global theorem for weighted projective threefolds with $A$-type codimension-two singularities therefore produces a global $\QQ$-Gorenstein smoothing
$X\rightsquigarrow Y$ with $Y$ smooth.

Set $L_0:=\mathcal O_X(42)$. Then $L_0^{[2]}\simeq\omega_X^{-1}$.
$L_0$ is not a line bundle on $X$, but it is a line bundle on the canonical-covering stack of $X$.

For a cyclic quotient $P\in X\simeq \frac1r(a,b,c)$, let
$q\equiv a+b+c\pmod r,\ e=\gcd(r,q).$ The index-one cover is obtained by quotienting by the kernel of the canonical character. This kernel has order $e$.
The reflexive sheaf $\mathcal O_X(d)$ pulls back to a line bundle on the index-one cover precisely when $e\mid d.$
For $d=42$, the four coordinate points give:
\[\begin{array}{c|c|c|c}
\text{point}&r&q& e=\gcd(r,q)\\ \hline
P_3&3&14+18+49\equiv0&3\\
P_{14}&14&3+18+49\equiv0&14\\
P_{18}&18&3+14+49\equiv12&6\\
P_{49}&49&3+14+18\equiv35&7
\end{array}\]
All four numbers
$3,\ 14,\ 6,\ 7$ divide $42$. Thus $L_0$ becomes invertible on every local canonical cover.

Therefore, if
$p_0:\mathfrak X_0\longrightarrow X$ is the canonical-covering stack, then
$\mathscr L_0:=p_0^{[*]}\mathcal O_X(42)$ is a line bundle on $\mathfrak X_0$.
Let $\pi:\mathcal X\longrightarrow\Delta$ be a $\QQ$-Gorenstein smoothing. Such a deformation determines a deformation
$\mathfrak X\longrightarrow\Delta$ of the canonical-covering stack; this is the stack interpretation of $\QQ$-Gorenstein deformations, as developed in \cite{Hac04}.
Hence $H^i(\mathfrak X_0,\mathcal O_{\mathfrak X_0})
 \simeq H^i(X,\mathcal O_X).$ Weighted projective space is arithmetically Cohen–Macaulay, so
$H^1(X,\mathcal O_X)=H^2(X,\mathcal O_X)=0.$ It follows that $\mathscr L_0$ extends uniquely through every infinitesimal thickening. After shrinking the disc, formal existence and algebraization give a line bundle
$\mathscr L\in\operatorname{Pic}(\mathfrak X)$ extending $\mathscr L_0$.
Moreover, on the central fiber,
$\mathscr L_0^{\otimes2}\simeq\omega_{\mathfrak X_0}^{-1}.$
Both $\mathscr L^{\otimes2}$ and $\omega_{\mathfrak X/\Delta}^{-1}$ extend this same line bundle. The uniqueness coming from $H^1(\mathcal O)=0$ therefore gives
$\mathscr L^{\otimes2}\simeq\omega_{\mathfrak X/\Delta}^{-1}.$

On a smooth fiber $Y$, the canonical-covering stack is simply $Y$. Thus
$H:=\mathscr L|_Y$ is a line bundle satisfying
$-K_Y=2H.$ It is ample  and $H^3=\frac{(-K_Y)^3}{8}=\frac{16}{8}=2.$
By the classification of Fano threefolds,  $Y\simeq V_2$ because $Y$ has Fano index 2 and $(-K_Y)^3=16$.
\end{proof}

\begin{Prop}\label{Y6}
The weighted projective space $X=\PP(5,6,9,10)$ admits a $\QQ$-Gorenstein smoothing to $Y_6$ a Fano threefold with $g=6$.
\end{Prop}

\begin{proof}
Let $X=\PP(5,6,9,10).$ Then
$-K_X=\mathcal O_X(30)$ and $(-K_X)^3
=\frac{30^3}{5\cdot6\cdot9\cdot10}=10.$ Its nontrivial gcd graph, after reordering the weights as
$(5,10,6,9)$, is
\[5\xleftrightarrow{\ 5\ }10
\xleftrightarrow{\ 2\ }6
\xleftrightarrow{\ 3\ }9.\]
The corresponding transverse singularities are respectively
$A_4,\ A_1,\ A_2.$
Use coordinates $[x:y:z:t]$, $\deg(x,y,z,t)=(5,10,6,9).$ Take invariants under the subgroup $\mu _5$. Set
\[u=z^5,\qquad v=zt,\qquad W=t^5.\]
After dividing all invariant degrees by $5$,
$X\simeq\{uw=v^5\}\subset\PP(1,2,3,6,9).$
Consider the degree-$15$ family
\[\mathcal X=
\left\{uw-v^5-s(x^{15}+xy^7)=0\right\}
\subset\PP(1,2,3,6,9)\times\mathbb A^1_s \tag{1}\]
with central fiber  $ \cong \PP(5,6,9,10)$.
For $s\ne0$, the fiber is defined by
\[X_{15}:=\{uw-v^5-x^{15}-xy^7=0\}\]
which is quasismooth. In particular, (1) completely smooths the $A_4$-curve corresponding to $\gcd(5,10)=5$.
Moreover, $K_{X_{15}}
=\mathcal O_{X_{15}}(-6)$,
so (1) is a $\QQ$-Gorenstein partial smoothing.

There remain two singular curves.
First,
\[C_2=\{x=v=w=0\}\simeq\PP(2,6),\]has generic transverse type $A_1$ with two end points $(P_u\in X_{15})\simeq\frac{1}{6}(1,2,3)$ and $(P_y\in X_{15})\simeq\frac{1}{2}(1,0,1)$.

Second,
\[C_3=\{x=y=0,\ uw=v^5\}\subset\PP(3,6,9)\]has generic transverse type $A_2$ with two end points $(P_u\in X_{15})\simeq\frac{1}{6}(1,2,3)$ and $(P_w\in X_{15})\simeq\frac{1}{9}(1,2,3)$. After rigidifying the generic $\mu_3$-stabilizer,
$C_3\simeq\PP(2,3),$ via
\[[z:t]\longmapsto[u:v:w]=[z^5:zt:t^5].\]The curves $C_2$ and $C_3$ meet only at the coordinate point $P_U$.
At that point, $\partial F/\partial w=u\ne0$, so $w$ can be eliminated. Hence $(P_u \in X_{15})\simeq\frac1{6}(1,2,3)$. This is precisely the anticanonical affine cone over $\PP(1,2,3)$.
Brown-Reid-Stevens \cite{BRS} prove that this cone has two genuine smoothing components, described by the Tom and Jerry unprojection. Thus the intersecting $A_1$ and $A_2$-branches in $P_u$  admit compatible simultaneous local smoothings. 

The other special points are also locally smoothable:
At $P_y$, the germ is $\frac12(1,0,1)\simeq\mathbb A^1\times A_1.$ At $P_w$, the germ is $\frac19(1,2,3).$ Its index-one cover is $\mathbb A^1\times A_2=\{uv=w^3\}.$
The deformation  $uv=w^3+\tau$ is invariant under the $\mu _3$-action. For $\tau\ne0$, the cover is smooth and the action is free, so this gives a local $\QQ$-Gorenstein smoothing.
Thus all local singularities of $X_{15}$ have compatible $\QQ$-Gorenstein smoothing directions.

Now, we need to show two vanishings: $H^2(X_{15},T_{X_{15}})=0$ and $H^1(X_{15},\mathcal T^1_{\mathrm{QG}})=0$.

For the first, the homogeneous coordinate ring of \(X_{15}\) is Cohen–Macaulay, so
\[H^1(X_{15},\mathcal O_{X_{15}}(m))
=H^2(X_{15},\mathcal O_{X_{15}}(m))=0\]for every $m$. The weighted Euler sequence and the hypersurface normal sequence
\[0\to T_{X_{15}}\to T_{\PP(1,2,3,6,9)}|_{X_{15}}
\to\mathcal O_{X_{15}}(15)\to0\]then give $H^2(X_{15},T_{X_{15}})=0$. 

For $\mathcal T^1_{\mathrm{QG}}$, away from $P_u$ its torsion-free restrictions are
\[\mathcal T^1_{\mathrm{QG}}|_{C_2}
\simeq\mathcal O_{\PP(1,3)}(4) \qquad \text{and} \qquad
\mathcal T^1_{\mathrm{QG}}|_{C_3}
\simeq\mathcal O_{\PP(2,3)}(10)
\oplus\mathcal O_{\PP(2,3)}(15).\]
Along $C_2$, the transverse coordinates have weights $1,3$, and the $A_1$-invariant equation has degree $2(1+3)=8.$ After quotienting by $\mu _2$, this becomes $\mathcal O(4)$. Along $C_3$, the transverse coordinates have weights $1,2$. For $A_2$, the two deformation coefficients have degrees
$3(1+2)=9,\ 2(1+2)=6.$ Since $\mathcal O_{X_{15}}(3)|_{C_3}
  \simeq\mathcal O_{\PP(2,3)}(5)$, these become $\mathcal O(15)$ and $\mathcal O(10)$.
All the line bundles in the above have vanishing $H^1$. The local Tom/Jerry description of $\frac16(1,2,3)$ \cite{BRS} shows that the correction at $P_u$ is zero-dimensional and that the endpoint evaluation maps from the two positive $C_3$-summands are surjective. Consequently,
$H^1(X_{15},\mathcal T^1_{\mathrm{QG}})=0.$ Therefore, the actual local smoothing components give a global $\QQ$-Gorenstein smoothing of $X_{15}$.

We now have a chain of smoothings
\[\PP(5,6,9,10)
\rightsquigarrow X_{15}
\rightsquigarrow Y,\]where $Y$ is smooth and $(-K_Y)^3=10$.

\medskip
\n {\bf Claim. } $\rho(Y)=1$.
\begin{proof}[Proof of the Claim] Let
$Z=X_{15}=\{uw=v^5+x^{15}+xy^7\}\subset\PP(1,2,3,6,9).$ Its singular locus is $C_2\cup C_3,$ where
$C_2\simeq\PP(1,3),\ C_3\simeq\PP(2,3),$ with transverse types $A_1$ and $A_2$. They meet at
$P_u, \ (P_u\in Z)\simeq\frac16(1,2,3).$ Let
$\pi:\mathcal Z\to\Delta$ be a sufficiently general global smoothing constructed from the prescribed local smoothing directions.
Weighted Lefschetz gives $H^2(Z,\QQ)\simeq\QQ.$ By Lemma~\ref{rho1}, it is enough to prove $\mathbb H^2(Z,R\phi_\pi\QQ)=0$.

Thus we must exclude global sections of the degree-two vanishing-cycle sheaf.
On $C_2$ we have $\mathcal T^1_{\mathrm{QG}}|_{C_2}
\simeq\mathcal O_{\PP(1,3)}(4).$ If $[y:u]$ have weights $(1,3)$, a general section is
\[q_4=\alpha y^4+\beta yu
=y(\alpha y^3+\beta u),
\qquad \alpha\beta\neq0.\]
It necessarily vanishes at $P_u=[0:1]$, where the special Brown-Reid-Stevens calculation \cite{BRS} is used. Away from $P_u$, it has one additional simple zero
$\alpha y^3+\beta u=0.$ By Lemma~\ref{monodromy},
$A_{1,\QQ}^{\mathrm{mon}}=0$.
Along $C_3$, $AB=C^3+s_1C+s_0,$ with
$s_1\in H^0(C_3,\mathcal O(10)),\
s_0\in H^0(C_3,\mathcal O(15)).$ At the weight $9$ endpoint $P_W$, the index-one cover is
$\{AB=C^3\}\times\mathbb A^1,$ and the selected smoothing is
$AB=C^3+\tau.$ The $\mu_3$-action acts on the $A_2$ root lattice as a $3$-cycle. Again by Lemma~\ref{monodromy},
$A_{2,\QQ}^{\mu_3}=0.$ For general $s_1,s_0$, one can also see this from the cubic root monodromy: its group is $S_3$, and
$A_{2,\QQ}^{S_3}=0$.

Locally,
$(P_u\in Z)\simeq Q:=\frac16(1,2,3),$ the anticanonical affine cone over
$S=\PP(1,2,3).$ Brown-Reid-Stevens \cite{BRS} construct two smoothing components of $Q$: Tom and Jerry. Their paper also identifies the corresponding $A_1$- and $A_2$-Weyl symmetries. The map from the stalk at $P_u$ to the nearby branch stalks is injective: \[\mathcal H^2(R\phi_\pi\QQ)_{P_u}\longrightarrow A_{1,\QQ}\oplus A_{2,\QQ}\simeq
\QQ\oplus\QQ^2.\]

Let $\sigma\in H^0\!\left(Z,\mathcal H^2(R\phi_\pi\QQ)\right).$ Its restriction to $C_2$ is zero because of the $A_1$ reflection monodromy. Its restriction to $C_3$ is zero because the residual $\mu_3$-monodromy, or equivalently the global $S_3$-monodromy, has no invariant vectors.
Therefore $\sigma$ can only be supported at $P_u$. But the injectivity of the local generization map at $P_u$ implies
$\sigma_{P_u}=0.$ Hence
$H^0\!\left(Z,\mathcal H^2(R\phi_\pi\QQ)\right)=0$ and therefore
$\mathbb H^2(Z,R\phi_\pi\QQ)=0.$
\end{proof}
Finally,
$(-K_Y)^3=10$. It follows easily that $Y$ has Fano index $1$, and
$g(Y)=\frac{(-K_Y)^3}{2}+1=6.$ Therefore $Y\simeq Y_6$.
\end{proof}

\begin{Rmk}\label{P(3)}
We know $\PP(3,12,20,25)$ is a numerical candidate for $Y_7$.
Let $X=\PP(3,12,20,25)$ with coordinate $[x:y:z:t]$. Its gcd graph is
$3\ \xleftrightarrow{\ 3\ }\ 12
\ \xleftrightarrow{\ 4\ }\ 20
\ \xleftrightarrow{\ 5\ }\ 25,$ and the middle curve is
\[C_{4}=\{x=t=0\}\simeq \PP(3,5),\]with transversal singularity $A_3$.

Take invariants for the group $\mu _4$. Put
$u=x^4,\ v=xt,\ w=t^4.$ After dividing the degrees by $4$,
$\deg(y,z,u,v,w)=(3,5,3,7,25)$. 
Consider a flat $\QQ$-Gorenstein family,
\[\mathcal X=\left\{uw-v^4-s\bigl(yz^5+y^7v\bigr)=0\right\}
\subset\PP(3,5,3,7,25)\times\mathbb A^1_s.\]

The central fiber is 
\[X \simeq \{uw=v^4\}
\subset \PP(3,5,3,7,25).\]
Indeed, one can check that $X_s$ is quasismooth for $s\ne0$. In particular, the middle $A_3$-curve can be smoothed completely.

Let $X_{28}:=X_{s=1}$. One does have
$H^2(X_{28},T_{X_{28}})=0.$ However, the $A_4$-curve contributes
$\dim H^1(X_{28},\mathcal T^1_{\mathrm{QG}})=6;$

The $A_2$-curve is
$L_3=\{z=v=w=0\}\simeq\PP(3,3).$ Along $L_3$, the linearization of the defining equation is
$u\,dw-y^7\,dv.$ Hence
\[0\to N_{L_3/X_{28}}\to\mathcal O(5)\oplus\mathcal O(7)\oplus\mathcal O(25)\xrightarrow{(0,-y^7,u)}\mathcal O(28)\to0.\]
The $\mathcal O(5)$-summand lies in the kernel, while
$\ker\bigl(\mathcal O(7)\oplus\mathcal O(25)\to\mathcal O(28)\bigr)
\simeq\mathcal O(4).$ Thus
\[N_{L_3/X_{28}}\simeq\mathcal O(5)\oplus\mathcal O(4),
\qquad\det N_{L_3/X_{28}}\simeq\mathcal O(9).\]
Since $\mathcal O_{\PP(3,3)}(3)\simeq\mathcal O_{\PP^1}(1)$, Lemma~\ref{T^1} gives
$\mathcal T^1_{\mathrm{QG}}|_{L_3}\simeq
\mathcal O_{\PP^1}(9)\oplus
\mathcal O_{\PP^1}(6)$, so $H^1(L_3,\mathcal T^1_{\mathrm{QG}}|_{L_3})=0.$

The $A_4$-curve is $L_5=\{y=u=v=0\}\simeq\PP(5,25).$ Along $L_5$, the linear part of the equation is
$w\,du-z^5\,dy.$ Therefore,
\[0\to N_{L_5/X_{28}}\to
\mathcal O(3)^{\oplus2}\oplus\mathcal O(7)\xrightarrow{(-z^5,w,0)}
\mathcal O(28)\to0.\]
The $\mathcal O(7)$-summand is contained in the kernel, and
$\ker\bigl(\mathcal O(3)^{\oplus2}\to\mathcal O(28)\bigr)
\simeq\mathcal O(-22).$ Hence
\[N_{L_5/X_{28}}\simeq\mathcal O(7)\oplus\mathcal O(-22),
\qquad\det N_{L_5/X_{28}}\simeq\mathcal O(-15).\]
Since $\PP(5,25)\rightsquigarrow\PP(1,5),
\ \mathcal O_{\PP(5,25)}(-15)
=\mathcal O_{\PP(1,5)}(-3).$ Lemma~\ref{T^1} gives
\[\mathcal T^1_{\mathrm{QG}}|_{L_5}
\simeq\mathcal O_{\PP(1,5)}(-15)
\oplus\mathcal O_{\PP(1,5)}(-12)
\oplus\mathcal O_{\PP(1,5)}(-9)
\oplus\mathcal O_{\PP(1,5)}(-6).\]
Since $K_{\PP(1,5)}=\mathcal O_{\PP(1,5)}(-6),$ Using Serre duality $h^1(L_5,\mathcal T^1_{\mathrm{QG}}|_{L_5})=6$.
Because \(L_3\) and \(L_5\) are disjoint, we have 
$H^1(X_{28},\mathcal T^1_{\mathrm{QG}})
\cong\mathbb C^6.$ This does not prove that $X_{28}$ is not smoothable, but is hard to check whether $X_{28}$ is smoothable.
\end{Rmk}

\begin{Prop}\label{Y10}
The weighted projective space \(X=\PP(2,18,25,45)\) admits a \(\QQ\)-Gorenstein smoothing to $Y_{10}$ a Fano threefold with \(g=10\).
\end{Prop}

\begin{proof}
Let \(X=\PP(2,18,25,45)\) with coordinates \(x_0,x_1,x_2,x_3\). Take invariants for the \(\mu_{9}\)-action on \(x_0,x_2\):
$u=x_0^9,\ v=x_0x_2,\ w=x_2^9.$
After dividing all degrees by $9$, we have $\deg(u,v,w,x_1,x_3)=(2,3,25,2,5).$
The invariant relation is $uw=v^9.$ Therefore
\[\PP(2,18,25,45)\simeq
\{uw=v^9\}\subset\PP(2,3,25,2,5).\]Consider the family
\[\mathcal X_t=
\left\{uw=v^9+tg_{27}(x_1,x_3)\right\}\subset\PP(2,3,25,2,5),\]
where $g_{27}(x_1,x_3)=\alpha x_1^{11}x_3+\beta x_1^6x_3^3+\gamma x_1x_3^5$ with general nonzero $\alpha,\beta,\gamma$. Each term has weighted degree $27$.
The family is flat and $\QQ$-Gorenstein, and
$\mathcal X_0\simeq\PP(2,18,25,45).$ For $t\neq0$,
\[F:=uw-v^9-tg_{27}(x_1,x_3)\]has derivatives
$F_u=w,\ F_w=u,\ F_v=-9v^8,$ together with $-tg_{x_1},-tg_{x_3}$. A general weighted binary form $g_{27}$ above has no multiple zero on $\PP(2,5)$, including its endpoints. 
Hence the affine cone is smooth away from its vertex, so the general $\mathcal X_t$ is quasismooth.

Denote a general nonzero fibre by
\[Z=X_{27}\subset\PP(2,2,3,5,25).\]
The remaining singular locus consists of two disjoint curves:
$L_2=\{v=x_3=w=0\}\simeq\PP(2,2)\simeq\PP^1$ 
with two end points $(P_u\in X_{27})\simeq\frac{1}{2}(1,0,1)$ and $(P_{x_1}\in X_{27})\simeq\frac{1}{2}(0,1,1)$, and
$L_5=\{u=v=x_1=0\}\simeq\PP(5,25)\simeq\PP(1,5)$
with two end points $(P_{x_3}\in X_{27})\simeq\frac{1}{5}(2,3,0)$ and $(P_w\in X_{27})\simeq\frac{1}{25}(3,2,5)$.
Their generic transverse types are $L_2:\ A_1$, and $L_5:\ A_4.$
Thus, the deformation changes the gcd graph $A_1$-$A_8$-$A_4$ to $A_1$ and $A_4$ without connecting edge.

Adjunction gives $-K_Z=\mathcal O_Z(10)$. On $L_2\simeq\PP(2,2)\simeq\PP^1,$ we have
$\mathcal O_Z(2)|_{L_2}\simeq\mathcal O_{\PP^1}(1),$ so
$-K_Z|_{L_2}\simeq\mathcal O_{\PP^1}(5)$. If $N_1,N_2$ are the two eigennormal bundles on the index-one cover, adjunction gives
\[N_1\otimes N_2\simeq K_{L_2}\otimes(-K_Z)|_{L_2}\simeq
\mathcal O_{\PP^1}(-2+5)=\mathcal O_{\PP^1}(3).\]
For an $A_1$-singularity, the smoothing parameter is a section of
$(N_1\otimes N_2)^{\otimes2}.$ Therefore
$\mathcal T^1_{\mathrm{QG}}|_{L_2}
\simeq\mathcal O_{\PP^1}(6).$ Choose a general section
$q_6\in H^0(\PP^1,\mathcal O(6))$ with six distinct zeros away from the endpoints. Locally the deformation is
$AB=C^2+\tau q_6.$ At each zero of $q_6$, $dq_6\neq0$, so the total general fiber remains smooth. Also, $H^1(\PP^1,\mathcal O(6))=0.$

Along $L_5$, the generic transverse singularity is $A_4$:
$y_0y_1=y_2^5.$ Here $y_2$ has weighted degree $5$, while the relation has degree $25$. Its miniversal form is
\[y_0y_1=y_2^5+s_3y_2^3+s_2y_2^2+s_1y_2+s_0,\]
where, on $L_5\simeq\PP(5,25)$,
$\deg(s_3,s_2,s_1,s_0)=(10,15,20,25).$ After dividing the curve weights by $5$, $L_5\simeq\PP(1,5),$ and therefore
$\mathcal T^1_{\mathrm{QG}}|_{L_5}
\simeq\mathcal O(2)\oplus\mathcal O(3)\oplus
\mathcal O(4)\oplus\mathcal O(5).$
It suffices to use only the constant parameter. Take
\[s_3=s_2=s_1=0,\qquad s_0=\lambda x_3^5+\mu w, \qquad \lambda\mu\ne0.\]
This is a section of $\mathcal O_{\PP(1,5)}(5)$. On the coarse $\PP^1$, $x_3^5$ and $w$ are homogeneous coordinates, so $s_0$ has one simple zero and is nonzero at both endpoints.
The local equation
\[y_0y_1=y_2^5+\tau(\lambda x_3^5+\mu w)\]is smooth for $\tau\neq0$: at the unique zero of $s_0$, its derivative along the base curve is nonzero.
Furthermore,
\[H^1\!\left(\PP(1,5),
\mathcal O(2)\oplus\mathcal O(3)\oplus
\mathcal O(4)\oplus\mathcal O(5)\right)=0.\]
At the weight-$25$ endpoint of $L_5$, the local singularity of $Z$ is
$\frac1{25}(3,2,5).$ Its canonical index is $5$, and its index-one cover is
$\frac15(3,2,0)\simeq\{y_0y_1=y_2^5\}\times\mathbb A^1.$ The deformation
$y_0y_1=y_2^5+\tau$ is smooth. The $\mu _5$-action has weights
$(3,2,1,1)$ on $(y_0, y_1, y_2,s)$. Its only fixed point is the origin, and the origin does not belong to the fiber $y_0y_1=y_2^5+\tau$ when $\tau\neq0$. Therefore the quotient general fiber is smooth.
Our choice $\mu w\neq0$ restricts precisely to this local smoothing.

Because $L_2\cap L_5=\varnothing$,
\[\mathcal T^1_{\mathrm{QG}}\simeq
i_{2*}\mathcal O_{\PP^1}(6)
\oplus i_{5*}\!\left(
\mathcal O(2)\oplus\mathcal O(3)
\oplus\mathcal O(4)\oplus\mathcal O(5)\right),\]
up to punctual terms at the coordinate points. Therefore, 
$H^1(Z,\mathcal T^1_{\mathrm{QG}})=0.$ Since $Z$ is a quasismooth weighted hypersurface, the weighted Euler and normal sequences give
$H^2(Z,T_Z)=0.$
The local-to-global deformation spectral sequence therefore globalizes the prescribed local smoothing directions. The only possible punctual obstruction is at the weight-$25$ point, and the index-one cover calculation above shows that the selected local deformation is unobstructed.
Consequently, $Z$ admits a projective $\QQ$-Gorenstein smoothing to a smooth Fano threefold $Y$. Adjunction gives
$-K_Z=\mathcal O_Z(10).$ Hence
$(-K_Z)^3=\frac{10^3\cdot27}{2\cdot2\cdot3\cdot5\cdot25}
=\frac{27000}{1500}=18.$ Therefore
$(-K_Y)^3=18.$

\n {\bf Claim.} \(\rho(Y)=1\).
\begin{proof}[Proof of the Claim]
Let $Z=X_{27}\subset \PP(2,2,3,5,25)$ be the partial smoothing constructed in the proof:
\[Z=\{uw=v^9+g_{27}(x_1,x_3)\}.\]
The remaining singular locus consists of two disjoint curves
$L_2\simeq\PP^1,\ L_5\simeq\PP(1,5)$,with transverse types $A_1$ and $A_4$. By Lemma~\ref{rho1}, it is enough to prove
$\mathbb H^2(Z,R\phi_\pi\QQ)=0.$
Because $L_2\cap L_5=\varnothing$, the contributions from $L_2$ and $L_5$ may be considered separately.

The smoothing along $L_2\simeq\PP^1$ is
$y_0y_1=y_2^2+\tau q_6,\ q_6\in H^0(\PP^1,\mathcal O(6)),$ where $q_6$ has six distinct simple zeros.
By Lemma~\ref{monodromy}, 
$H^0\!\left(L_2, \mathcal H^2(R\phi_\pi\QQ)\right)=0.$
The section $s_0$ has one simple zero, and the five roots of
$C^5+\tau s_0=0$ are cyclically permuted. Then by Lemma~\ref{monodromy},
$H^0\!\left(L_5,\mathcal H^2(R\phi_\pi\QQ)\right)=0.$

At the weight-$25$ endpoint, the singularity is
$\frac1{25}(3,2,5).$ Its index-one cover is
$\frac15(3,2,0)\simeq
\{y_0y_1=y_2^5\}\times\mathbb A^1.$ The chosen deformation lifts to
$y_0y_1=y_2^5+\tau.$  The $\mu_5$-action has weights
$(3,2,1,1)$ on $(y_0,y_1,y_2,s)$, and acts freely on the local general fiber. The Milnor fiber of the cover has
\[H^0(F,\QQ)=\QQ,\qquad
H^2(F,\QQ)\simeq\QQ^4,\]
and no other reduced rational cohomology. Hence
$\chi(F)=5.$ Since the $\mu_5$-action is free,
$\chi(F/\mu_5)=\frac{\chi(F)}5=1,$ and we have 
$H^*(F/\mu_5,\QQ)
\simeq H^*(F,\QQ)^{\mu_5}.$ The quotient is connected, so its $H^0$ already contributes $1$ to the Euler characteristic. It follows that
$H^2(F,\QQ)^{\mu_5}=0.$ Thus the weight-$25$ endpoint contributes no punctual degree-two vanishing cycle. Both curve contributions and the endpoint contribution vanish:
$H^0\!\left(Z,\mathcal H^2(R\phi_\pi\QQ)\right)=0.$ Therefore
$\mathbb H^2(Z,R\phi_\pi\QQ)=0.$ The nearby-cycle sequence now gives
$b_2(Y)=1, \ \rho(Y)=1.$
\end{proof}

Finally, $(-K_Y)^3=(-K_Z)^3=18$. If the Fano index were $r\ge2$, then
$(-K_Y)^3=r^3H^3$ would lead to a contradiction. Thus the index is $1$, and $g(Y)=\frac{(-K_Y)^3}{2}+1=10.$ Hence the smooth fibre is a prime Fano threefold $Y_{10}$ of genus $10$.
\end{proof}

\begin{Prop}\label{Y12}
The weighted projective space $X=\PP(2,9,22,33)$ admits a $\QQ$-Gorenstein smoothing to  $Y_{12}$ a Fano threefold with $g=12$.
\end{Prop}

\begin{proof}
Let $X=\PP(2,9,22,33)$ with coordinates $x_0,x_1,x_2,x_3$. Take invariants for the $\mu_{11}$-action on $x_0,x_1$:
$u=x_0^{11},\ v=x_0x_1,\ w=x_1^{11},\ x=x_2,\ y=x_3.$
After dividing all degrees by $11$,
$\deg(v,u,x,y,w)=(1,2,2,3,9),$  and
$uw=v^{11}.$ Hence
\[X\simeq\{uw=v^{11}\}\subset\PP(1,2,2,3,9).\]
Now consider
\[\mathcal X=\left\{uw=v^{11}+t(x^4y+xy^3)\right\}
\subset\PP(1,2,2,3,9) \times \mathbb{A}^1_t.\]
It is a $\QQ$-Gorenstein flat family, and $\mathcal X_0\simeq\PP(2,9,22,33)$.

The deformation term
\[x^4y+xy^3\in H^0\!\left(\PP(2,3),\mathcal O(11)\right)\]
is exactly the $s_0$-parameter for the $A_{10}$-curve $C_{23}$.
For $t\neq0$, put
\[F=uw-v^{11}-t(x^4y+xy^3).\]
Then $F_u=w,\ F_w=u,\ F_v=-11v^{10}, \ F_x=-ty(4x^3+y^2), \text{and} \ F_y=-tx(x^3+3y^2).$
If all partial derivatives vanish, then
$u=w=v=0$ and $y(4x^3+y^2)=0,\ x(x^3+3y^2)=0$. These equations imply $x=y=0$. Thus $\mathcal X_t$ is quasismooth.

Its singular locus consists of two disjoint curves:
$L_2=\{v=y=w=0\}\simeq\PP(2,2),$ with generic transverse type $A_1$, and $L_3=\{v=u=x=0\}\simeq\PP(3,9),$ with generic transverse type $A_2$. Consequently, the original gcd graph 
$A_1-A_{10}-A_2$ has been changed into two vertices $A_1$ and $A_2$ without connecting edge. 

We let
\[X_{11}=\{uw=v^{11}+x^4y+xy^3\}\subset \PP(1,2,2,3,9),\]
where
$\deg(v,u,x,y,w)=(1,2,2,3,9)$. $L_2$ has two end points $(P_u\in X_{11})\simeq\frac{1}{2}(1,0,1)$ and $(P_x\in X_{11})\simeq\frac{1}{2}(1,0,1)$. $L_3$ has wo end points $(P_y\in X_{11})\simeq\frac{1}{3}(1,2,0)$ and $(P_w\in X_{11})\simeq\frac{1}{9}(1,2,3)$. Now we give a $\QQ$-Gorenstein smoothing of $X_{11}$ to a smooth Fano threefold of genus $12$.

By an explicit computation, $\mathcal T^1_{\mathrm{QG}}|_{L_2}
\simeq\mathcal O_{\PP^1}(2).$ Hence an $A_1$-smoothing is determined by a quadratic $q_2(u,x)\in H^0(\PP^1,\mathcal O(2)).$ Choose $q_2$ with two distinct zeros, neither at a coordinate endpoint. Locally the deformation has the form
$y_0y_1=y_2^2+\tau q_2.$ At a zero of $q_2$, the derivative $dq_2$ is nonzero, so the total general fiber is smooth.
Moreover,
\[H^1\bigl(L_2,\mathcal T^1_{\mathrm{QG}}|_{L_2}\bigr)
=H^1(\PP^1,\mathcal O(2))=0.\]

Along $L_3$, the generic transverse singularity is $A_2$. In miniversal form it is
\[y_0y_1=y_2^3+s_1y_2+s_0.\]Here $y_2$ has weighted degree $3$, while the relation has degree $9$. Therefore
\[s_1\in H^0\!\left(\PP(3,9),\mathcal O(6)\right),
\qquad s_0\in H^0\!\left(\PP(3,9),\mathcal O(9)\right).\]
Since $L_3\simeq\PP(1,3),$ $\mathcal T^1_{\mathrm{QG}}|_{L_3}
\simeq \mathcal O_{\PP(1,3)}(2)
\oplus\mathcal O_{\PP(1,3)}(3).$
In particular,
$H^1\bigl(L_3,\mathcal T^1_{\mathrm{QG}}|_{L_3}\bigr)=0.$ Take
\[s_1=a\,y^2,\qquad s_0=b\,y^3+c\,w\]with $abc\ne0$. The discriminant of $C^3+s_1C+s_0$ is 
\[\Delta=-4s_1^3-27s_0^2=-4a^3y^6-27(by^3+cw)^2.\]
Putting $Y=y^3,\ W=w,$ this becomes a quadratic form
$\Delta(Y,W)=-4a^3Y^2-27(bY+cW)^2.$ For general $a,b,c$, it has two distinct zeros on the coarse $\PP^1$, and
\[\Delta(1,0)=-4a^3-27b^2\ne0,\qquad
\Delta(0,1)=-27c^2\ne0.\]
Thus the discriminant avoids both coordinate endpoints. At its two zeros, $d\Delta\ne0$, so the total threefold remains smooth.

There is one non-Gorenstein singular point $P_w$ at $L_3$, which is $\frac1{9}(1,2,3)$. Its index-one cover is $\frac13(1,2,0)\simeq \{y_0y_1=y_2^3\}\times\mathbb A^1.$ The deformation $y_0y_1=y_2^3+\tau$ is smooth for $\tau\ne0$. The $\mu _3$-action has weights
$(1,2,1,1)$ on $(y_0,y_1,y_2,s)$; its only fixed point is the origin, which does not lie on $y_0y_1=y_2^3+\tau$. Hence the quotient general fiber is smooth. Our choice
\[s_0=b\,y^3+c\,w,\qquad c\ne0,\]restricts precisely to this local smoothing. Thus all local smoothing directions, including the codimension-three endpoint, are compatible.

Since \(L_2\cap L_3=\varnothing\),
$\mathcal T^1_{\mathrm{QG}}
\simeq i_{2*}\mathcal O_{\PP^1}(2)
\oplus i_{3*}\left(
\mathcal O_{\PP(1,3)}(2)
\oplus\mathcal O_{\PP(1,3)}(3)\right).$
Therefore $H^1(X_{11},\mathcal T^1_{\mathrm{QG}})=0$.
It remains to check $H^2(X_{11},T_{X_{11}})=0$. The weighted Euler sequence restricted to $X_{11}$ is
\[0\to\mathcal O_{X_{11}}
\to\mathcal O_{X_{11}}(1)\oplus\mathcal O_{X_{11}}(2)^{\oplus2}
\oplus\mathcal O_{X_{11}}(3)\oplus\mathcal O_{X_{11}}(9)
\to T_{\PP}|_{X_{11}}\to0.\]Since $X_{11}$ is a quasismooth weighted hypersurface, it is arithmetically Cohen–Macaulay, so the required intermediate cohomology of these line bundles vanishes. The normal sequence is
\[0\to T_{X_{11}}\to T_{\PP}|_{X_{11}}\to\mathcal O_{X_{11}}(11)\to0.
\] It follows that
$H^2(X_{11},T_{X_{11}})=0.$
This shows that the chosen sections on $L_2$ and $L_3$ globalize and extend to an actual formal $\QQ$-Gorenstein smoothing. 

Adjunction gives $-K_X=\mathcal O_X(6)$, and $(-K_X)^3=\frac{6^3\cdot11}{1\cdot2\cdot2\cdot3\cdot9}=22.$
The anticanonical divisor remains ample in the smoothing, so the general fiber $Y$ is a smooth Fano threefold with $(-K_Y)^3=22.$

\medskip
\n {\bf Claim.} $\rho(Y)=1$.
\begin{proof}[Proof of the Claim]
By Lemma~\ref{rho1}, it is enough to prove $\mathbb H^2(X_{11},R\phi_\pi\QQ)=0$.
The complex $R\phi_\pi\QQ$ is supported on the singular locus of $X_{11}$, $L_2$ and $L_3$, where the generic transverse singularities are $A_1$ and $A_2$, respectively.
Since $\mathbb H^2(X_{11},R\phi_\pi\QQ)
\simeq H^0\!\left(X_{11},\mathcal H^2(R\phi_\pi\QQ)\right)$, we must show that there is no global monodromy invariant section of the $A_1$ or $A_2$ vanishing-cycle local systems.

Along $L_2\simeq\PP^1\simeq\PP(1,1)$, the smoothing has local form
\[y_0y_1=y_2^2+\tau q_2,
\qquad q_2\in H^0(\PP^1,\mathcal O(2)),\]
where $q_2$ has two simple zeros. By Lemma~\ref{monodromy},
$H^0(L_2,\mathcal H^2(R\phi_\pi\QQ))=0.$

Along $L_3\simeq\PP(1,3),$ the smoothing is
\[y_0y_1=y_2^3+s_1y_2+s_0,\qquad
s_1=ay^2,\quad s_0=by^3+cw.\]
On the chart $y\neq0$, put
$\lambda=\frac{w}{y^3}.$ After rescaling $y_2$, the relevant cubic is
\[p_\lambda(y_2)=y_2^3+ay_2+b+c\lambda.\] For general $a,b,c$, the monodromy group is the full symmetric group $S_3$. Then by Lemma~\ref{monodromy},
$H^0(L_3,\mathcal H^2(R\phi_\pi\QQ))=0.$

At the non-Gorenstein endpoint of $L_3$, the singularity is $\frac1{9}(1,2,3).$ Its index-one cover is $\frac13(1,2,0)\simeq\{y_0y_y1=y_2^3\}\times\mathbb A^1.$ The $\mu_3$-action acts on the $A_2$ vanishing lattice by a $3$-cycle. Thus $A_{2,\QQ}^{\mu_3}=0$ by Lemma~\ref{monodromy}. So the endpoint does not create an additional punctual degree-two invariant.
Combining the two disjoint curves and the endpoint calculation gives
$H^0\!\left(X_{11},\mathcal H^2(R\phi_\pi\QQ)\right)=0,$ hence
$\mathbb H^2(X_{11},R\phi_\pi\QQ)=0.$
\end{proof}

The index cannot be at least $2$, since that would make $(-K_Y)^3$ divisible by $2^3=8$, whereas $(-K_Y)^3=22.$ Thus $Y$ has index $1$, and $g(Y)=\frac{(-K_Y)^3}{2}+1=12.$
Therefore $X_{11}\subset\PP(1,2,2,3,9)$
admits a $\QQ$-Gorenstein smoothing to $Y_{12}$.
Combining this with the previous partial smoothing gives
$\PP(2,9,22,33)\rightsquigarrow X_{11}\rightsquigarrow Y_{12}.$
\end{proof}



\begin{thebibliography}{12}



\bibitem{BR86}
M.~Beltrametti, L.~Robbiano,
\emph{Introduction to the theory of weighted projective
spaces}, Exposition. Math., 4(2) (1986) 111--162.

\bibitem{BRS}
G.~D.~Brown, M.~Reid, J.~Stevens, 
\emph{Tutorial on Tom and Jerry: the two smoothings of the anticanonical cone over \(\PP(1,2,3)\)}, 
EMS Surveys in Mathematical Sciences 8,  no. 1-2  (2021), 25--38.

\bibitem{CL26}
J.~Chen, Y.~Lee,
\emph{Weighted projective spaces admitting \(\QQ\)-Gorenstein smoothings to \(\PP^3\)}, Preprint.

\bibitem{Del73}
P.~Deligne,
\emph{Le formalisme des cycles évanescents, Exposé XIII, and La formule de Picard–Lefschetz, Exposé XV}, in Groupes de monodromie en géométrie algébrique II, LNM 340, Springer, 1973.


\bibitem{DeV22}
K.~DeVleming,
\textit{Moduli of surfaces in $\PP^3$}, 
Compositio Math. \textbf{158} (2022), 1329-1374.


\bibitem{DLT}
K.~DeVleming, J.~Li, and S.~Torres,
\textit{Weighted projective degenerations of $\PP^n$},
arXiv:2606.21660.

\bibitem{Dimca}
A.~ Dimca,
\emph{Sheaves in Topology}, 
Universitext, Springer, 2004


\bibitem{Dol82}
I.~Dolgachev,
\emph{Weighted projective varieties}. In Group actions and vector fields (Vancouver, B.C.,
1981), volume 956 of Lecture Notes in Math., pages 34--71. Springer, Berlin, 1982.

\bibitem{Gal}
S.~Galkin,
\emph{Small toric degenerations of Fano threefolds}, arXiv:1809.02705.

\bibitem{Hac04}
P.~Hacking, 
\emph{Compact moduli of plane curves}, Duke Math. J. 124 (2004), no. 2, 213--257.

\bibitem{HP10}
P.~Hacking, Y.~Prokhorov,
\emph{Smoothable del Pezzo surfaces with quotient singularities},
Compos. Math., 146(1) (2010), 169--192. 

\bibitem{IF00}
A.R.~Iano-Fletcher,
\emph{Working with weighted complete intersections}, 
In Explicit birational geometry of 3-folds, volume 281 of London Math. Soc. Lecture Note Ser., 101--173. Cambridge Univ.
Press, Cambridge, 2000.

\bibitem{ILP}
N.~O~Ilten, J.~Lewis, V.~Przyjalkowski, 
\emph{Toric degenerations of Fano threefolds giving weak Landau-Ginzburg models},
 J. Algebra 374 (2013), 104--121.

\bibitem{Max20}
L.~G.~Maxim, \emph{Notes on vanishing cycles and applications}, J. Aust. Math. Soc. 109  no.3 (2020), 371--415.

\bibitem{AGV}
A.~N.~Parshin, I.~R.~Shafarevich,
\emph{Algebraic Geometry. V. Fano varieties},  A translation of Algebraic geometry. 5 (Russian), Ross. Akad. Nauk, Vseross. Inst. Nauchn. i Tekhn. Inform., Moscow. Translation edited by A. N. Parshin and I. R. Shafarevich. Encyclopaedia of Mathematical Sciences, 47. Springer-Verlag, Berlin, 1999.

\bibitem{Sch}
M.~Schlessinger,
\emph{Rigidity of quotient singularities}, Invent. Math., 14 (1971) 17--26. 






\end{thebibliography}
\end{document}